\documentclass[10pt, oneside]{article}
\usepackage[T1]{fontenc}\usepackage[utf8]{inputenc}
\usepackage{mathrsfs}
\usepackage{amsfonts}
\usepackage{amssymb}
\usepackage{amsmath,amsfonts,amssymb,amsthm}

\usepackage{caption}
\usepackage{subfigure}
\usepackage{cite}
\usepackage{color}
\usepackage{graphicx}
\usepackage{subfigure}
\usepackage{float}
\usepackage{paralist}
\usepackage{indentfirst}
\usepackage{authblk}
\usepackage[colorlinks]{hyperref}
\AtBeginDocument{%
   \hypersetup{%
    linkcolor=blue,%
    citecolor=black,%
  }%
}

\usepackage[normalem]{ulem}

\usepackage[title]{appendix}

\definecolor{red}{rgb}{1.00,0.00,0.00}
\definecolor{blue}{rgb}{0.00,0.00,0.63}
\definecolor{black}{rgb}{0.00,0.00,0.00}
\definecolor{purple}{rgb}{0.00,1.00,0.00}
\definecolor{pink}{rgb}{0.95,0.01,0.08}

\usepackage[T1]{fontenc}

\newtheorem{theorem}{Theorem}[section]
\newtheorem{lemma}{Lemma}[section]
\newtheorem{prop}{Proposition}[section]

\newtheorem{remark}{Remark}[section]

\newtheorem{defn}{Definition}[section]

\def\var{\varepsilon}

\def\bma#1\ema{{\allowdisplaybreaks\begin{aligned}#1\end{aligned}}}

\numberwithin{equation}{section}

\begin{document}

\author{Qingyou He, Ling-Yun Shou \& Leyun Wu}

\title{{\LARGE \textbf{Hele-Shaw limit of  chemotaxis-Navier-Stokes flows}}}


\date{}
\maketitle
\begin{abstract}
This paper investigates the connection between the chemotaxis-Navier-Stokes system with  porous medium type nonlinear diffusion  and the Hele-Shaw problem in $\mathbb{R}^d$ ($d\geq2$). First, we prove the global-in-time existence of weak solutions for the Cauchy problem of the chemotaxis-Navier-Stokes system with general initial data, uniformly in the diffusion range $m\in [3,\infty)$. In particular, this existence result does not require additional structural assumptions on the chemotactic sensitivity $\chi(c)$ and the oxygen consumption rate $f(c)$, and remains valid in arbitrary dimensions. Then, we rigorously justify the Hele-Shaw limit for this system as $m\rightarrow\infty$, showing the convergence to a free boundary problem of Hele-Shaw type, where the bacterium (cell) diffusion is governed by the stiff pressure law. Moreover, the complementarity relation characterizing the limiting bacterium (cell) pressure via a degenerate elliptic equation is verified by a novel application of the Hele-Shaw framework.
\end{abstract}

\noindent{\bf Keywords.} Chemotaxis-Navier-Stokes system;  Nonlinear diffusion; Global existence; Hele-Shaw problem.
\\
2020 {\bf MSC.} 35A01; 35B40; 35B44; 35K55; 76D27; 92C17.

\section{Introduction}



In biological environments, bacteria (cells) typically inhabit viscous fluids, where both the bacteria (cells) and the chemical signals they produce are transported by the surrounding medium. Moreover, the fluid dynamics may be affected by gravitational forces resulting from bacterium (cell) aggregation. A notable example is the bacterium {\emph{Bacillus subtilis}} suspended in water, where experiments have demonstrated the spontaneous emergence of spatial patterns from an initially homogeneous distribution. To simulate bacterium (cell)–fluid interactions, Tuval et al. \cite{PANS2005} first  introduced the chemotaxis-Navier–Stokes system in $\mathbb{R}^d$ ($d\geq2$):
\begin{equation}\label{eq1-1}
\left\{
\begin{aligned}
&\partial_t n+u\cdot \nabla n=\Delta n^m-
\nabla \cdot (n\chi(c) \nabla c),\\
&\partial_t c+u \cdot \nabla c=\Delta c-nf(c),\\
&\partial_t u +(u \cdot \nabla) u+\nabla \Pi=\Delta u-n\nabla \phi,\\
&\nabla \cdot u=0.
	\end{aligned}
 \right.
\end{equation}
Here $n=n(t,x): \mathbb{R}^{+}\times\mathbb{R}^d\rightarrow \mathbb{R}^{+}$, $c=c(t,x):\mathbb{R}^{+}\times\mathbb{R}^d\rightarrow \mathbb{R}^+,$ $u=u(t,x):\mathbb{R}^{+}\times\mathbb{R}^d\rightarrow \mathbb{R}^d$ and $\Pi=\Pi(t,x):\mathbb{R}^{+}\times\mathbb{R}^d\rightarrow  \mathbb{R}$ denote the density of the bacterium (cell) population, the concentration of the oxygen (chemotactic signal), the fluid velocity and the pressure of the fluid, respectively.  $m>1$ is the parameter associated with porous medium type nonlinear slow diffusion.  $\chi(c)$ and $f(c)$ stand for the chemotactic sensitivity and the oxygen consumption rate, respectively.  $\phi=\phi(x)$ denotes the  potential function produced by different physical mechanisms. We are concerned with the Cauchy problem for \eqref{eq1-1} supplemented with the initial data
\begin{align} 
(n,c,u)(x,0)=(n_0,c_0,u_0)(x),\quad \quad x\in \mathbb{R}^d.\label{d}
\end{align}
Note that the motion of the bacterium (cell) density can be described by a conservative equation
\begin{align}
\partial_t n+\nabla\cdot ( nV)=0,\label{mass}
 \end{align}
coupled with Darcy's law
\begin{align}
    V=u-\nabla P+\chi(c)\nabla c,\label{darcy}
\end{align}
where  $V:\mathbb{R}^{+}\times\mathbb{R}^d\to\mathbb{R}^d$ is the velocity field for bacterium (cell) motion, and 
\begin{align}
P:=\frac{m}{m-1}n^{m-1}\label{Pin}
\end{align}
is the  bacterium (cell) pressure.

The main purpose of this work is to establish a new bacterium (cell) diffusion mechanism, namely, the stiff pressure law, in the context of the chemotaxis-fluid interaction system \eqref{eq1-1} as the diffusion exponent $m\to\infty$. More precisely, the stiff pressure, solving a degenerate elliptic equation, exists in the bacterium (cell) saturation region where the bacterium (cell) density is equal to $1$ and suppresses bacterium (cell) density larger than $1$. To this end, we first verify the global existence of weak solutions for the Cauchy problem \eqref{eq1-1}-\eqref{d} with $m\geq3$ (see Theorem~\ref{thm1}). Furthermore, we establish the Hele-Shaw (or incompressible) limit  as the diffusion exponent $m\to\infty$  (see Theorem~\ref{thm2}). This limit solves a free boundary problem of Hele-Shaw type with a complementarity relation. These findings extend the mathematical theory of coupled chemotaxis-fluid systems and offer new insights into modeling interactions between biological and fluid dynamics under nonlinear diffusion or in Hele-Shaw flow regimes.


\vspace{2mm}




\noindent
{\textbf{Chemotaxis-fluid equations: Linear diffusion.}} 
The classical chemotaxis-Navier-Stokes equations with linear diffusion (i.e., \eqref{eq1-1} with $m=1$) have been studied extensively, yielding many significant results. Regarding the Cauchy problems, Duan, Lorz and Markowich \cite{duan2010} established the global well-posedness and convergence rates of classical solutions for the chemotaxis–Navier–Stokes equations in $\mathbb{R}^3$, provided that the initial perturbation is small in $H^3(\mathbb{R}^3)$. Furthermore, considering the chemotaxis-Stokes system (i.e., \eqref{eq1-1} without $u\cdot\nabla u$ and with $m=1$) instead of the chemotaxis–Navier–Stokes system in $\mathbb{R}^2$, the authors \cite{duan2010} demonstrated the global existence of weak solutions to the corresponding Cauchy problem under either weak external forcing or small substrate concentration.  
Subsequently, Liu and Lorz \cite{liu2011} removed the previous weak external forcing or small substrate concentration and obtained global weak solutions for chemotaxis-Navier-Stokes equations under the structural conditions
\begin{equation}
\begin{aligned}
\chi(c),f(c),f'(c),\chi'(c)\geq0,\quad f(0)=0,\quad \frac{\chi'(c)f(c)+\chi(c)f'(c)}{\chi(c)}>0,\quad \frac{d^2}{dc^2}\Big(\frac{f(c)}{\chi(c)}\Big)<0,\label{struct2}
\end{aligned}
\end{equation}
and the initial assumptions
\begin{equation}\label{initial1}
\begin{aligned}
n_0(1+|x|+|\log n_0|)\in L^1(\mathbb{R}^d),\quad c_0\in L^1(\mathbb{R}^d)\cap L^{\infty}(\mathbb{R}^d),\quad \nabla \Psi(c_0)\in L^2(\mathbb{R}^d),\quad u_0\in L^2(\mathbb{R}^d),
\end{aligned}
\end{equation}
where $\Psi(c)$ is defined by $\Psi(c):=\int_0^c\sqrt{\frac{\chi(s)}{f(s)}}\,ds$.  Chae, Kang and Lee \cite{ckl2014} presented some blow-up criteria for the
local solutions to the chemotaxis-Navier-Stokes system in $\mathbb{R}^d$ ($d=2,3$) and proved the global existence of classical solutions in $\mathbb{R}^2$ under quite different assumptions from \eqref{struct2} that for some constant $\nu$,
\begin{equation}
\begin{aligned}
\chi(c),f(c),f'(c),\chi'(c)\geq0, \quad \sup_{c}|\chi(c)-\nu f(c)|\ll 1.\label{struct3}
\end{aligned}
\end{equation}
Duan, Li and Xiang \cite{rxz2017} obtained the existence and uniqueness of weak solutions and classical solutions in the two-dimensional case when $\|n_0\|_{L^1(\mathbb{R}^2)}$ is suitably small. Lorz  \cite{lorz2012} showed the global existence of weak solutions in $\mathbb{R}^3$ with small $L^{\frac{3}{2}}$-norm of initial data. The uniqueness of weak solutions has been addressed by Zhang and Zheng; cf.~\cite{zz2014}. Chae, Kang and Lee \cite{ckl2013} obtained the local existence of regular solutions with $(n_0,u_0)\in H^{s}(\mathbb{R}^d)$ and $c_0\in H^{s+1}(\mathbb{R}^d)$ with $s\geq3$ and $d=2,3$. In \cite{ckl2014},  they also presented some blow-up criteria and constructed global solutions for the three-dimensional chemotaxis-Stokes equations under some smallness  of initial data. For the Cauchy problem of the self-consistent chemotaxis-fluid system, Carrillo, Peng and Xiang  \cite{cpx2023} established several extensibility criteria of classical solutions in two and three dimensions. In the same paper, they also presented the global weak solution with small $c_0$ for the three-dimensional flow. We also refer to the recent work \cite{heshouwu} where the authors relaxed the smallness condition for global well-posedness and time decay such that the possibly large oscillations are allowed. 

When the spatial domain is considered to be a bounded domain with a smooth boundary, the local existence of weak solutions is established by Lorz \cite{Lotz10}. Winkler \cite{MR2876834} obtained a unique global classical solution in the two-dimensional case and proved the global existence of weak solutions in the three-dimensional Stokes case when $\chi(c)=1$ and $f(c)=c$, which cannot be covered by  \eqref{struct2} due to $\frac{d^2}{dc^2}\big(\frac{f(c)}{\chi(c)}\big)=0$. Such solutions of the two-dimensional version, shown in \cite{wm2014,Zhang2d},  converge to a unique spatially homogeneous steady state  at an exponential rate as the time tends to infinity. The global solvability of weak solutions to the three-dimensional chemotaxis-Navier–Stokes
system was obtained by Winkler \cite{wm2016} as the limit of smooth solutions for suitably regularized problems. We also refer to \cite{px2018,px2019} concerning the global existence or stabilization of solutions to the initial boundary value problems in different domains. Such global weak solutions
do become smooth eventually but may develop singularities prior to such ultimate regularization (see \cite{MR2876834,wm2017}). Recently, Winkler \cite{wm2023JEMS} established Leray’s structure theorem to characterize the possible extent of unboundedness phenomena.

Recent studies revealed that fluid flows can significantly influence the bacterium (cell) aggregation behaviors for Keller–Segel type equations, where the oxygen equation is governed by an elliptic type Poisson equation. For instance, Kiselev and Ryzhik \cite{Kiselev1} analyzed the impact of specific fluid flows on spreading properties and additional absorbing reactions in broadcast spawning models. Moreover, Kiselev and Xu \cite{Kiselev3} developed a framework in which introducing a suitably chosen incompressible velocity field via a simple transport mechanism prevented the blow-up phenomena that would otherwise occur in the classical Keller–Segel system. 
Additionally, enhanced dissipation and blow-up suppression in the two-dimensional chemotaxis-fluid systems near the Couette flow were investigated by Zeng, Zhang and Zi \cite{ZZZ-2021-JFA} and He \cite{He-2023-SIAM}. Lai, Wei and Zhou \cite{LWZ2023} showed the global existence of free-energy solutions for the two-dimensional system with critical and subcritical mass $8\pi$. Recently, the impressive work  \cite{HuKiselevYao1} showed that  the bacterium (cell) aggregation was suppressed via buoyancy in the general fluid context.

\noindent
{\textbf{Chemotaxis-fluid equations: Nonlinear diffusion.}} When $m>1$, the nonlinear bacterium (cell) diffusion mechanisms in \eqref{eq1-1} cause essential mathematical difficulties due to the degeneracy of $\Delta n^m$ near $n=0$. Under $m\in (\frac{3}{2},2]$, Francesco, Lorz and Markowich \cite{map2010} constructed global   weak solutions for the chemotaxis-Navier-Stokes system \eqref{eq1-1} in either a bounded domain or the whole space in $\mathbb{R}^2$. In the three-dimensional case, they also established a similar global existence result for the chemotaxis-Stokes system. Using \emph{a priori} estimates derived from the Lyapunov functional under the conditions \eqref{struct2} and \eqref{initial1}, Liu and Lorz \cite{liu2011} improved the range of $m$ to $(\frac{4}{3},2]$. With the same assumptions, Duan and Xiang \cite{DuanIMRN12} addressed the optimal condition on   $m>1$ ensuring global existence.  Tao and Winkler \cite{ym2012} constructed global weak solutions in a two-dimensional bounded domain for $\chi(c)=1$, $f(c)=c$ and arbitrary $m>1$. They \cite{wm2013AIHP} further established the global existence and large-time asymptotics of locally bounded solutions to the initial boundary-value problem in three dimensions and studied its large-time asymptotics for $m>\frac{8}{7}$. Then, the global existence and boundedness of weak solutions for the initial boundary value problem were justified by Winkler \cite{wm2015} and Zhang and Li \cite{ZhangLi15}. One of the key ingredients in \cite{liu2011,DuanIMRN12} is the use of the functional
\begin{equation}
\begin{aligned}
E(t):=\int_{\mathbb{R}^d} \Big( n\log{n}+n\sqrt{1+|x|^2} +\frac{1}{2}|\nabla\Psi(c)|^2+\frac{1}{2}|u|^2\Big)\,dx,\label{entropy}
\end{aligned}
\end{equation}
under the conditions \eqref{struct2} and \eqref{initial1}. Here the space-weighted term $n\sqrt{1+|x|^2} $ plays a role in ensuring the lower bound of the functional $E(t)$ due to the $n\log{n}$ term.

A natural question is whether the global existence of weak solutions holds for suitably large $m$ and general $\chi(c), f(c)$ without the structural condition \eqref{struct2} and the spatial weight assumption for $n_0$.

\vspace{1mm}

\noindent
{\textbf{Hele-Shaw limit.}} 
The Hele-Shaw (incompressible) limit for the Patlak-Keller-Segel model (with Newtonian attractive potential) was first established in \cite{CKY2018} using a combination of viscosity solution and gradient flow. Recently,  for the Keller-Segel system, even in the presence of a growth term, general attractive kernel and volume-filling effect, the Hele-Shaw limits were proved via weak solution techniques; cf.~\cite{HLP2023,HLP2022,HZ2024}. Perthame et al.  \cite{PERTHAME2014} first studied the Hele-Shaw asymptotics for the porous medium type reaction-diffusion equation modeling tumor growth, which leads to a significant body of research in this direction~\cite{DAVID2021,DPSV2021,DS2020,GKM2022,KM2023,KT2018,LX2021,PV2015}. Representative studies on the Hele-Shaw limit for tumor growth models using weak solutions were conducted in~\cite{DAVID2021,PERTHAME2014,PQTV2014}. The incompressible (Hele-Shaw) limit for tumor growth incorporating convective effects was rigorously analyzed in \cite{DS2021}, and the decay rates on the diffusion exponent \( m \) were further explored in \cite{DDB2022,NAF2024}.  For tumor growth models governed by Brinkmann’s pressure law,  the convergence for density and pressure  was established in \cite{KT2018} through viscosity solution methods. The authors   \cite{DPSV2021,DS2020} proved the Hele-Shaw limit for the two-species case via compactness techniques. The Hele-Shaw asymptotics for porous medium equations with non-monotonic or non-local reaction terms were obtained through the viewpoint of the obstacle problem in \cite{GKM2022}. Furthermore, for tissue growth incorporating autophagy, the existence of weak solutions and the Hele-Shaw limit  were analyzed in \cite{LX2021}.  In the case of porous medium equations with drift, the singular limit was studied using viscosity solution methods in \cite{KPW2019}. The convergence of the free boundary in the incompressible limit of tumor growth with convective effects was recently achieved in \cite{ZT2014}. In addition, the non-symmetric traveling wave solutions and the rigorous derivation for a Hele-Shaw type tumor growth model with nutrient supply were provided in \cite{FHLZ2024}.  

To the best of our knowledge, there is no result concerning both the uniform regularity estimates with respect to $m$ and the rigorous justification of the Hele-Shaw limit for \eqref{eq1-1}.

\vspace{2mm}

\noindent
{\textbf{Our contributions.}} 
In the present paper, we build a new bacterium (cell) diffusion mechanism (stiff pressure law) in the context of chemotaxis-fluid interaction equations \eqref{eq1-1}. 
\begin{itemize}
\item On the one hand, we prove a new global existence theorem for the chemotaxis-Navier-Stokes system \eqref{eq1-1} with any $m\in[3,\infty)$. Different from earlier significant works {\rm{\cite{liu2011,rxz2017}}}, our result relaxes the conditions \eqref{struct2} and \eqref{struct3} for $\chi(c), f(c)$ and replaces the space-weighted assumption $n_0(1+|x|+|\log n_0|)\in L^1(\mathbb{R}^d)$ in \eqref{initial1} by $n_0\in L^{m-1}(\mathbb{R}^d)$. In particular, we obtain the result on the global existence of weak  solutions in the presence of the fluid convection effect $u\cdot\nabla u$ in \eqref{eq1-1} for arbitrary  dimensions $d\geq2$, while previous studies focused on the chemotaxis-Stokes system when $d=3$. Towards this end, inspired by \eqref{mass} and \eqref{darcy}, we establish \emph{a priori} estimates using the energy functional
\begin{equation}
\begin{aligned}
\mathcal{E}(t):=\int_{\mathbb{R}^d} \Big(\frac{1}{m-2}P +\frac{1}{2}|\nabla c|^2+\frac{1}{2}|u|^2\Big)\,dx,\label{entropynew}
\end{aligned}
\end{equation}
where  the effective ``pressure'' $P:=\frac{m}{m-1}n^{m-1}$ is used to replace the $n\log{n}$ term in \eqref{entropy}.  In addition, we show that the corresponding a priori estimates are uniform in $m\geq3$.  
\item On the other hand, our work provides the first rigorous justification of the Hele-Shaw limit for the complex chemotaxis-fluid interaction flows and finds a novel approach to verify the complementarity relation. Specifically, departing from classical methods, we introduce a special test function acting on the established Hele–Shaw system \eqref{eq1-4}-\eqref{ilr} to justify this relation. This particularly allows us to weaken the spatially weighted assumptions required in the analysis such as Lemma \ref{ka}.

To highlight the novelty of this approach, we briefly review existing strategies. The authors \cite{CKY2018} combined viscosity solution with gradient flow  to establish the Hele–Shaw limit for the Keller–Segel model with Newtonian potential when the initial data is a patch function. This result was later extended to the same model with general initial data via a weak solution framework in \cite{HLP2022}. Specifically, the estimates in \cite{HLP2022} such as the Aronson–B\'enilan estimate, the $L^1$ estimate for the time derivative of the pressure and the $L^3$ estimate of the pressure gradient collectively yield the $L^2$ strong convergence of the pressure gradient, thereby providing a sufficient condition for the complementarity relation. More recently, for chemotaxis systems with growth or volume-filling effect, inspired by tissue growth models \cite{noemi2023, LX2021}, the authors \cite{HLP2023,HZ2024} leveraged the special structure of the porous medium type equation to achieve $L^2$ strong convergence of the pressure gradient to verify the complementarity relation, without the required additional regularities as in \cite{HLP2022}.  In addition, as a comparison, we also derive the complementarity relation  (see Appendix~\ref{appendix1}) by additionally proving the $L^2$ strong compactness of the gradient of the $m$-th power density $n_m^m$ via exploiting the porous medium type structure as in \cite{HLP2023,HZ2024}. 
\end{itemize}


\noindent\textbf{Organization of this article.} In the forthcoming section \ref{sect:result}, we state our main results (Theorems \ref{thm1} and \ref{thm2}). The proof of Theorem \ref{thm1} is presented in Section \ref{section2}, which consists of the construction of approximate solutions and the compactness argument. In Section \ref{sec:hs}, we first establish the additional uniform regularity estimates of global weak solutions and then justify rigorously the Hele-Shaw limit  as $m\rightarrow\infty$.  Furthermore, we verify the complementarity relation via the obtained Hele-Shaw framework. Another proof of the complementarity relation based on the compactness techniques \cite{HLP2023,HZ2024} is carried out in Appendix \ref{appendix1}. 
Appendix~\ref{App:proofapp} presents some supplementary  calculations for the part of weak solutions (Section \ref{section2}).
Some useful technical results are collected in Appendix \ref{appendix2}.

\section{Main results}\label{sect:result}

Before stating our main results, we give the definition of global weak solutions. Throughout this paper, for any time $T>0$, we denote  
$$Q_T:=\mathbb{R}^d\times (0,T).$$

\begin{defn}[Weak solution]\label{ws}
 A triple $(n ,c ,u )\in L_{\rm loc}^1(Q_T)$ is called a global weak solution for the Cauchy problem \eqref{eq1-1}-\eqref{d} of the chemotaxis-Navier-Stokes system with the initial data $(n_0 ,c_0 ,u_0 )\in L_{\rm loc}^1(\mathbb{R}^d)$ if for any given time $T>0$, the following properties hold{\rm:}  
 \begin{itemize}
 \item $ n |u|,\ n ^m,\ \chi(c )n |\nabla c|,\ c |u|,\ nf(c),\ |u|^2$ are locally integrable in $Q_T$.
 \item For any scalar function $\varphi\in \mathcal{C}_0^\infty([0,T)\times\mathbb{R}^d)$ and  $d$-vector valued function $\psi\in [\mathcal{C}_0^\infty([0,T)\times\mathbb{R}^d)]^d$ satisfying $\nabla\cdot \psi=0$, we have 
\begin{align*}
&\iint_{Q_T}\big(n \partial_t \varphi+n  u \cdot \nabla\varphi + n ^m\Delta\varphi +\chi(c)n \nabla c \cdot\nabla\varphi \big) \,dxdt+\int_{\mathbb{R}^d}n_0\varphi(0,x)\,dx=0,\\
&\iint_{Q_T}\big(c \partial_t \varphi+c  u \cdot \nabla\varphi+ c \Delta \varphi -n f(c ) \varphi\big) \,dxdt+\int_{\mathbb{R}^d}c_0\varphi(0,x)\,dx=0,\\
&\iint_{Q_T}\big(u \cdot \partial_t \psi+(u \otimes u) : \nabla \psi+u\cdot \Delta \psi-n \nabla\phi \cdot \psi\big) \,dxdt+\int_{\mathbb{R}^d}u_0\cdot \psi(0,x)\,dx=0,\\
&\iint_{Q_T}u \cdot\nabla \varphi\, dxdt=0.
\end{align*}

\end{itemize}
\end{defn}

\paragraph{Assumptions.} We suppose that $\chi,f$ and $\phi$ satisfy 
\begin{align}\label{ca}\tag{${\rm{H}}_1$}
\chi,\ f\in W^{1,\infty}(\mathbb{R}_+),\quad\quad f\geq0,\quad\quad  \phi\in W^{1,\infty}(\mathbb{R}^d),
\end{align}
and the initial data $(n_0,c_0,u_0)$ has the properties
\begin{equation}\label{a1}\tag{${\rm{H}}_2$}
\left\{
\begin{aligned}
& n_0,\ c_0 \geq0,\quad\quad\quad\quad\quad~~~  \nabla\cdot u_0=0,\\
&\|n_{0}\|_{L^{1}(\mathbb{R}^d)}\leq C_0,\quad\quad ~ \quad\|c_{0}\|_{L^1(\mathbb{R}^d)}\leq C_0,~~ \quad\quad\quad  c_0\leq c_B,\\
&\|n_{0}\|_{L^{m-1}(\mathbb{R}^d)}\leq C_0,\quad \quad ~\|c_0\|_{H^1(\mathbb{R}^d)}\leq C_0,\quad \quad\quad~ \|u_{0}\|_{L^2(\mathbb{R}^d)}\leq C_0,
\end{aligned}
\right.
\end{equation}
where $c_B$ and $C_0$ are two positive constants independent of $m$.

\begin{theorem}[Global existence]\label{thm1}
Let $d\geq 2$ and $m\geq3$, and assume that \eqref{ca} and \eqref{a1} hold. Then the Cauchy problem \eqref{eq1-1}-\eqref{d} admits a global weak solution $(n,c,u)$ with $n, c\geq0$ in the sense of Definition~\ref{ws} such that, for any time $T>0$,
\begin{equation}\label{r1}
\left\{
\begin{aligned}
&n\in L^{\infty}(0,T;L^1(\mathbb{R}^d)\cap L^{m-1}(\mathbb{R}^d)),\\
&c\in L^{\infty}(0,T;L^{1}(\mathbb{R}^d)\cap L^{\infty}(\mathbb{R}^d)\cap H^1(\mathbb{R}^d))\cap L^2(0,T; H^2(\mathbb{R}^d)),\\
& u\in L^{\infty}(0,T;L^{2}(\mathbb{R}^d))\cap L^{2}(0,T;H^{1}(\mathbb{R}^d)),
\end{aligned}
\right.
\end{equation}
and
\begin{equation}\label{basep}
\begin{aligned}
&\sup\limits_{t\in[0,T]} \mathcal{E}(t)+\int_0^T\Big(\|\nabla P(t)\|_{L^2(\mathbb{R}^d)}^2+\frac{1}{2}\|\Delta c\|_{L^2(\mathbb{R}^d)}^2+\frac{1}{2}\|\nabla u\|_{L^2(\mathbb{R}^d)}^2\Big) \,dt\\
&\quad \leq C\Big(\| c_0\|_{L^2(\mathbb{R}^d)}^2+\frac{1}{m-2}\|n_0\|_{L^{m-1}(\mathbb{R}^d)}^{m-1}+\|u_0\|_{L^2(\mathbb{R}^d)}^2+T\Big),
\end{aligned}
\end{equation}
where $P$ and $\mathcal{E}(t)$ are defined in \eqref{Pin} and \eqref{entropynew}, respectively, and $C>0$ is a constant independent of $m$. In particular, the corresponding a priori bounds in \eqref{r1} are independent of $m$. 
\end{theorem}



\begin{remark}
The fluid pressure $\Pi$ can be solved by the 
elliptic problem 
\begin{align}
   -\Delta \Pi=\nabla\cdot\nabla\cdot(u\otimes u)+\nabla\cdot(n\nabla\phi)\label{pressure} 
\end{align}
in the distributional sense. According to the regularity properties of $n$ and $u$ in \eqref{r1} {\rm(}see Proposition \ref{propuee}{\rm)}, using the elliptic regularity theory yields 
$$
\Pi\in L^{\frac{2(d-1)}{d}}(0,T;L^{\frac{d-1}{d-2}}(\mathbb{R}^d))\quad \text{for}\quad 
 d\geq3\quad  \text{and}\quad  \Pi\in L^{\frac{3}{2}}(0,T;L^3(\mathbb{R}^2))\quad  \text{for}\quad 
 d=2.
$$

\end{remark}

Next, we aim to study the Hele-Shaw limit as the diffusion exponent $m\to\infty$. To this end, we label $(n, c, u, P, \Pi)$ by $(n_m, c_m, u_m, P_m, \Pi_m)$ and rewrite the chemotaxis-Navier-Stokes system \eqref{eq1-1} as
\begin{equation}\label{eqm}
	\begin{cases}
\partial_t n_m+u_m\cdot \nabla n_m=\Delta n_m^m-
\nabla \cdot (n_m\chi(c_m) \nabla c_m),\\
\partial_t c_m+u_m \cdot \nabla c_m=\Delta c_m-n_mf(c_m),\\
\partial_t u_m +(u_m \cdot \nabla) u_m+\nabla \Pi_m=\Delta u_m-n_m\nabla \phi,\\
\nabla \cdot u_m=0,	\end{cases}
\end{equation}
with the initial data 
\[(n_m(x,0),c_{m}(x,0),u_{m}(x,0))=(n_{m,0}(x),c_{m,0}(x),u_{m,0}(x)),\quad x\in\mathbb{R}^d.\]
The  pressure \big($(m-1)$-th power of the density\big) expressed by $P_m:=\frac{m}{m-1}n_m^{m-1}$ plays a central role in the analysis. Indeed, $P_m$ satisfies the equation
\begin{equation}\label{dpe}
\partial_t P_m+u_m\cdot\nabla P_m=(m-1)P_m(\Delta P_m-\nabla \cdot(\chi(c_m)\nabla c_m))+\nabla P_m\cdot(\nabla P_m-\chi(c_m)\nabla c_m).
\end{equation}
Formally, we derive that  the limit $(n_\infty,P_\infty, c_\infty,u_\infty, \Pi_\infty)$ of $(n_m,P_m,c_m,u_m,\Pi_m)$ in $m$ solves a so-called Hele-Shaw type system 
\begin{equation}\label{eq1-4}
	\begin{cases}
\partial_t n_\infty+u_\infty\cdot \nabla n_\infty=\Delta P_\infty -
\nabla \cdot (n_\infty\chi(c_\infty) \nabla c_\infty),\\
\partial_t c_\infty+u_\infty \cdot \nabla c_\infty=\Delta c_\infty-n_\infty f(c_\infty),\\
\partial_t u_\infty +(u_\infty \cdot \nabla) u_\infty+\nabla \Pi_\infty=\Delta u_\infty-n_\infty\nabla \phi,\\
\nabla \cdot u_\infty=0,		
	\end{cases}
\end{equation}
with the initial data 
\begin{align}
(n_\infty,c_\infty, u_\infty)(0,x)=(n_{\infty,0},c_{\infty,0},u_{\infty,0})(x),\label{eq1-4:d}
\end{align}
and the following Hele-Shaw graph relation
\begin{equation}\label{ilr}
0\leq n_\infty\leq1,\quad (1-n_\infty)P_\infty=0,\quad (1-n_\infty)\nabla P_\infty=0.
\end{equation}
Equations \eqref{eq1-4} and \eqref{ilr} give a weak formulation of the Hele-Shaw type free boundary problem.
We need one more complementarity equation to describe the limiting pressure $P_\infty$. We take the limit for the pressure equation~\eqref{dpe} in $m$ and formally derive a degenerate elliptic equation, called the {\emph{complementarity relation}}: 
\begin{equation}\label{cr2}
P_\infty(\Delta P_\infty -\nabla\cdot(\chi(c_\infty)\nabla c_\infty))=0.
\end{equation}

In addition to \eqref{ca} and \eqref{a1}, we impose the following additional assumptions:
\begin{equation}\label{H3}\tag{${\rm{H}}_3$}
\left\{
\begin{aligned}
&\|n_{m,0}\|_{L^{m+1}(\mathbb{R}^d)}^{m+1}\leq C,\\
&\lim_{m\rightarrow\infty}\big( \|n_{m,0}-n_{\infty,0}\|_{L^1(\mathbb{R}^d)}+\|c_{m,0}-c_{\infty,0}\|_{L^1(\mathbb{R}^d)}+\|u_{m,0}-u_{\infty,0}\|_{L^2(\mathbb{R}^d)}\big)= 0,
\end{aligned}
\right.
\end{equation}
where the constant $C>0$ is independent of $m$.

Then, we establish a rigorous justification of the convergence from the chemotaxis-Navier-Stokes system \eqref{eqm} to the  Hele-Shaw type system \eqref{eq1-4}-\eqref{ilr} with the complementarity property~\eqref{cr2} as $m\rightarrow \infty$.


\begin{theorem}[Hele-Shaw limit]\label{thm2}
Let $(n_m,c_m,u_m)$ be a weak solution for the Cauchy problem \eqref{eqm} obtained in Theorem \ref{thm1} with $m\geq \max\{d+1,5\}$, and let the fluid pressure $\Pi_m$ be given by \eqref{pressure}. Set $P_m:=\frac{m}{m-1} n_m^{m-1}$ as the bacterium (cell) pressure. In addition, we define $q\in (1,\infty)$, $p\in [1,\frac{2d}{d-2})$, $(p_1,q_1):=(\frac{2(d-1)}{d},\frac{d-1}{d-2})$ for $d\geq3$ and $(p_1,q_1)=(\frac{3}{2},3)$ for $d=2$. Then, under the assumptions \eqref{ca}, \eqref{a1} and \eqref{H3} on $\chi,\ f,\ \phi$ and the initial data $(n_{m,0},c_{m,0},u_{m,0})$, there exists a limit $(n_\infty,c_\infty,u_\infty,P_\infty,\Pi_\infty)$ such that as $m\rightarrow \infty$, it holds true that 
\begin{equation}\label{uml}
\left\{
\begin{alignedat}{4}
u_m \to u_\infty \quad
&\text{strongly}\quad &\text{in}\quad &L^2(0,T;L^2_{\mathrm{loc}}(\mathbb{R}^d)),\\
c_m \to c_\infty \quad
&\text{strongly}\quad &\text{in}\quad &L^2(0,T;W^{1,p}_{\mathrm{loc}}(\mathbb{R}^d)),\\
\Pi_m \rightharpoonup \Pi_\infty \quad
&\text{weakly}\quad &\text{in}\quad &L^{p_1}(0,T;L^{q_1}(\mathbb{R}^d)),\\
P_m \rightharpoonup P_\infty \quad
&\text{weakly}\quad &\text{in}\quad &L^2(0,T;H^1(\mathbb{R}^d)),\\
n_m^m \rightharpoonup P_\infty \quad
&\text{weakly}\quad &\text{in}\quad &L^2(0,T;H^1(\mathbb{R}^d)),\\
n_m \rightharpoonup n_\infty \quad
&\text{weakly-$\ast$}\quad &\text{in}\quad &L^\infty(0,T;L^q(\mathbb{R}^d)),\\
n_m \to n_\infty \quad
&\text{strongly}\quad &\text{in}\quad &L^2(0,T;\dot{H}^{-1}_{\mathrm{loc}}(\mathbb{R}^d)).
\end{alignedat}
\right.
\end{equation}
Moreover, the limit $(n_\infty,c_\infty, u_\infty,P_\infty,\Pi_\infty)$ satisfies the Hele-Shaw type system \eqref{eq1-4}-\eqref{eq1-4:d} in the sense of distributions, with the Hele-Shaw graph \eqref{ilr} almost everywhere, and the complementarity relation \eqref{cr2} remains valid in the distributional sense.
\end{theorem}

\begin{remark} In general, the free boundary of the Hele-Shaw problem \eqref{eq1-4}-\eqref{ilr} is referred to as the support boundary of the density or the pressure. When the initial cell mass $M(>0)$ is bounded, due to the conservation of mass, the saturation region, where the density equals $1$, is bounded. Since the pressure is supported on the level set of the density $1$, the pressure is compactly supported, which means the existence of the free boundary.
\end{remark}

\begin{remark}
In Theorem \ref{thm1}, the condition 
$m\geq 3$ is used to ensure the global existence of weak solutions. For Theorem \ref{thm2}, we impose the stronger condition $m\geq \max\{d+1,5\}$, which is needed to obtain 
$m$-independent higher integrability estimates (see Lemma \ref{lemma43}). These estimates play a key role in deriving time-derivative bounds (see Lemma \ref{lemma44}) and in the compactness argument for passing to the Hele–Shaw limit. Note that this condition is automatically satisfied in the singular limit 
$m\to \infty$.
\end{remark}

\begin{remark}
In Appendix~\ref{appendix1}, we give an alternative proof of the validity of the complementarity relation \eqref{cr2} by employing the classical method as in {\rm{\cite{HLP2023,HZ2024}}}, under the additional conditions that $\|n_{m,0}\|_{L^{m+3}(\mathbb{R}^d)}^{m+3}$, $\||x|^2n_{m,0}\|_{L^1(\mathbb{R}^d)}$ and $\||x|c_{m,0}\|_{L^2(\mathbb{R}^d)}$ are uniformly bounded with respect to $m$. In particular, we can obtain the strong convergence from $\nabla n_m^m$ to $\nabla P_\infty$ in $L^2(Q_T)$ as $m\rightarrow \infty$ {\rm{(}}see Proposition \ref{sg}{\rm)}.
\end{remark}

\section{Global existence}\label{section2}

This section is devoted to the proof of Theorem \ref{thm1} on the global existence of weak solutions. 
For any $0<\varepsilon<1$, we regularize the initial data as follows
\begin{align}
&(n_{0,\varepsilon},c_{0,\varepsilon}, u_{0,\varepsilon})(x):=(J_{\varepsilon}\ast n_{0},J_{\varepsilon}\ast c_{0},J_{\varepsilon}\ast u_{0})(x),\nonumber
\end{align}
where $J_{\varepsilon}$ denotes the mollifier. By the initial assumptions \eqref{a1}, one knows that $(n_{0,\varepsilon},c_{0,\varepsilon}, u_{0,\varepsilon})$ is smooth for any $0<\varepsilon<1$ and satisfies
\begin{equation}\label{duniform}
\begin{alignedat}{2}
&0\le n_{0,\varepsilon}, \qquad\qquad\qquad
&&\|n_{0,\varepsilon}\|_{L^p(\mathbb{R}^d)}\le \|n_0\|_{L^p(\mathbb{R}^d)}, \qquad 1\le p\le m-1,\\
&0\le c_{0,\varepsilon}\le c_B, \qquad
&&\|c_{0,\varepsilon}\|_{L^p(\mathbb{R}^d)}\le \|c_0\|_{L^p(\mathbb{R}^d)}, \qquad 1\le p\le \infty,\\
&\|\nabla c_{0,\varepsilon}\|_{L^2(\mathbb{R}^d)}\le \|\nabla c_0\|_{L^2(\mathbb{R}^d)}, \qquad
&&\|u_{0,\varepsilon}\|_{L^2(\mathbb{R}^d)}\le \|u_0\|_{L^2(\mathbb{R}^d)}.
\end{alignedat}
\end{equation}
We consider the following approximating equations with artificial viscosity and regularized aggregation: 
\begin{equation}\label{app}
\left\{
\begin{aligned}
&\partial_t n_\varepsilon+u_\varepsilon\cdot \nabla n_\varepsilon
=\Delta n_\varepsilon^m+\varepsilon \Delta n_\varepsilon
-\nabla \cdot \big(n_\varepsilon\chi(c_\varepsilon)
\nabla( J_{\varepsilon}\ast c_\varepsilon)\big),\\
&\partial_t c_\varepsilon+ u_\varepsilon \cdot \nabla c_\varepsilon
=\Delta c_\varepsilon-n_\varepsilon f(c_\varepsilon),\\
&\partial_t u_\varepsilon +u_\varepsilon \cdot \nabla u_\varepsilon+\nabla \Pi_\varepsilon
=\Delta u_\varepsilon-n_{\varepsilon}\nabla \phi,\\
&\nabla \cdot u_{\varepsilon}=0,
\end{aligned}
\right.
\end{equation}
with the initial data 
\begin{equation}
(n_{\varepsilon},c_{\varepsilon},u_{\varepsilon})(x,0)=(n_{0,\varepsilon},c_{0,\varepsilon}, u_{0,\varepsilon})(x).\label{appd}
\end{equation}

We have the following proposition pertaining to the global existence of the approximate sequence. For the proof, one can refer to Appendix \ref{App:proofapp}.
\begin{prop}\label{propapp}
For any fixed $0<\varepsilon<1$, there exists a global weak solution $(n_\varepsilon,c_\varepsilon,u_\varepsilon)$ to the approximate problem \eqref{app}-\eqref{appd} satisfying
\begin{equation}\label{r100}
\left\{
\begin{aligned}
&n_{\varepsilon}\in L^{\infty}(0,T;L^1(\mathbb{R}^d)\cap L^{p}(\mathbb{R}^d))\cap L^2(0,T;H^1(\mathbb{R}^d)),\\
&c_{\varepsilon}\in L^{\infty}(0,T;L^{1}(\mathbb{R}^d)\cap L^{\infty}(\mathbb{R}^d)\cap H^1(\mathbb{R}^d))\cap L^2(0,T; H^2(\mathbb{R}^d)),\\
& u_{\varepsilon}\in L^{\infty}(0,T;L^{2}(\mathbb{R}^d))\cap L^{2}(0,T;H^{1}(\mathbb{R}^d)), 
\end{aligned}
\right.
\end{equation}
for any time $T>0$ and $1<p<\infty$.
\end{prop}

To proceed, the key point is to establish the a priori estimates that are uniform with respect to both $\var$ and $m$. These estimates allow us to pass the limit as $\varepsilon\rightarrow0$ and prove the convergence of the global approximate sequence $\{(n_\varepsilon,c_\varepsilon,u_\varepsilon)\}_{0<\varepsilon<1}$ to the desired global weak solution $(n,c,u)$ of the Cauchy problem \eqref{eq1-1}-\eqref{d}.


\subsection{Uniform energy estimates}

This subsection concerns uniform regularity estimates of $(n_{\varepsilon},c_{\varepsilon},u_{\varepsilon})$. We first state low-order regularity estimates of $n_{\varepsilon}$ and $c_{\varepsilon}$.

\begin{lemma}\label{lemma21}
If $(n_{\varepsilon},c_{\varepsilon},u_{\varepsilon})$ is the strong solution to \eqref{app}-\eqref{appd} on $[0,T]\times\mathbb{R}^d$ for a given time $T>0$, then under the assumption of \eqref{ca} and \eqref{a1}, we have
\begin{align}
&\|n_{\varepsilon}(t)\|_{L^1(\mathbb{R}^d)}\leq \|n_0\|_{L^1(\mathbb{R}^d)},\quad t\in[0,T],\label{211}\\
&\|c_{\varepsilon}(t)\|_{L^1(\mathbb{R}^d)}\leq \|c_0\|_{L^1(\mathbb{R}^d)},\quad ~t\in[0,T],\label{212}\\
&\|\nabla c_{\varepsilon}\|_{L^2(Q_T)}\leq \|c_0\|_{L^2(\mathbb{R}^d)},\label{213}\\
&0\leq c_{\varepsilon}(t,x)\leq  c_B,\quad\quad \quad \quad\quad (t,x)\in [0,T]\times\mathbb{R}^d.\label{214}
\end{align}
\end{lemma}

\begin{proof}
Owing to the maximum principle for the first and second equations of $\eqref{app}$, one gets $n_\varepsilon\geq0$ and $0\leq c_\varepsilon\leq \|c_{0,\varepsilon}\|_{L^{\infty}(\mathbb{R}^d)}\leq c_B$, which verifies \eqref{214}. Then, integrating $\eqref{app}_1$ and $\eqref{app}_2$ in time gives rise to \eqref{211}-\eqref{212}. The estimate \eqref{213} can be achieved by taking the $L^2$ scalar product of $\eqref{app}_2$ with  $c_\var$ and using the facts that $\nabla\cdot  u_\varepsilon=0$ and $f\geq0$. We omit the details for brevity.
\end{proof}

In order to establish higher integrability estimates of $n_{\varepsilon}$ which are uniform in $\var$, our key ingredient is to introduce the effective ``pressure'' term  
\begin{align*}
P_\varepsilon=\frac{m}{m-1}n_{\varepsilon}^{m-1}.
\end{align*}

\begin{lemma}
Let $m>2$, and $T>0$ be any given time.
Under the assumptions \eqref{ca} and \eqref{a1}, we have 
\begin{equation}
\begin{aligned}
&\frac{1}{m-2}\sup_{t\in[0,T]}\|P_\varepsilon(t)\|_{L^1(\mathbb{R}^d)}+ \|\nabla P_\varepsilon\|_{L^2(Q_T)}^2\leq C\| c_0\|_{L^2(\mathbb{R}^d)}^2+\frac{C}{m-2}\|n_0\|_{L^{m-1}(\mathbb{R}^d)}^{m-1},\label{PL1var}
\end{aligned}
\end{equation}
and 
\begin{equation}
\begin{aligned}
\sup_{t\in[0,T]}\|n_\varepsilon(t)\|_{L^{p}(\mathbb{R}^d)}\leq N_0,\quad  1< p\leq m-1,\label{nm1var}
\end{aligned}
\end{equation}
where $C, N_0>0$ are two constants independent of $T$ and $m$.
\end{lemma}

\begin{proof}
The term $P_\var$ allows us to rewrite $\eqref{app}_1$ as
\begin{equation*}\label{dpevar}
\begin{aligned}
\partial_t P_\varepsilon+&u_\varepsilon\cdot\nabla P_\varepsilon\\
\quad=&(m-1)P_\varepsilon\big(\Delta P_\varepsilon-\nabla \cdot(\chi(c_\varepsilon)\nabla (J_{\varepsilon}\ast c_\varepsilon)\big)+\nabla P_\varepsilon\cdot\big(\nabla P_\varepsilon-\chi(c_\varepsilon)\nabla (J_\varepsilon*c_\varepsilon)\big)\\
&+\varepsilon  m n_\varepsilon^{m-2}\Delta n_\varepsilon.
\end{aligned}
\end{equation*}
Integrating on $Q_t$ and using the Cauchy-Schwarz inequality gives rise to
 \begin{equation*}
 \begin{aligned}
\int_{\mathbb{R}^d}P_{\varepsilon}\, &dx+(m-2)\iint_{Q_t}|\nabla P_\varepsilon|^2\,dxd\tau+\frac{4\varepsilon  m(m-2)}{(m-1)^2}\iint_{Q_t}|\nabla n_\varepsilon^{\frac{m-1}{2}}|^2\,dxd\tau\\
 =&\int_{\mathbb{R}^d}\frac{m}{m-1}n_{0,\varepsilon}^{m-1}\, dx+(m-2)\iint_{Q_t}\nabla P_\varepsilon\cdot\nabla( J_{\varepsilon}\ast c_\varepsilon)\chi(c_\varepsilon)\,dxd\tau\\
 \leq &\int_{\mathbb{R}^d}\frac{m}{m-1}n_{0,\varepsilon}^{m-1}\, dx+\frac{m-2}{2}\iint_{Q_t}|\nabla P_\varepsilon|^2\,dxd\tau+\frac{m-2}{2}\iint_{Q_t}\chi^2(c_\varepsilon)|\nabla( J_{\varepsilon}\ast c_\varepsilon)|^2\,dxd\tau.
 \end{aligned}
 \end{equation*}
 Here by \eqref{duniform}, \eqref{213}, \eqref{214} and $\|\nabla( J_{\varepsilon}\ast c_\varepsilon)\|_{L^2(\mathbb{R}^d)}\leq \|\nabla c_\varepsilon\|_{L^2(\mathbb{R}^d)}$ due to Young's inequality for convolutions, one gets
 \begin{equation*}
 \begin{aligned}
 \frac{2}{m-2}\sup\limits_{t\in[0,T]}&\|P_{\varepsilon}(t)\|_{L^1(\mathbb{R}^d)}+\|\nabla P_\varepsilon\|_{L^2(Q_T)}^2 \\
 \leq& \sup_{0\leq s\leq c_B}|\chi(s)|^2\|\nabla c_\varepsilon\|_{L^2(Q_T)}^2+\frac{2m}{(m-1)(m-2)}\|n_{0,\varepsilon}\|_{L^{m-1}(\mathbb{R}^d)}^{m-1}\\
 \leq& C \|c_0\|_{L^2(\mathbb{R}^d)}^2+\frac{C}{m-2}\|n_{0}\|_{L^{m-1}(\mathbb{R}^d)}^{m-1} .
 \end{aligned}
 \end{equation*}
 This leads to \eqref{PL1var}.

In addition, one infers from \eqref{a1} and \eqref{PL1var} that
\begin{equation*}
\begin{aligned}
\sup\limits_{t\in[0,T]}\|n_\varepsilon(t)\|_{L^{m-1}(\mathbb{R}^d)}&\leq \Big( (m-2)\|c_0\|_{L^2(\mathbb{R}^d)}^2+\|n_0\|_{L^{m-1}(\mathbb{R}^d)}^{m-1} \Big)^{\frac{1}{m-1}}\\
&\leq \Big( (m-2)C_0^2+C_0^{m-1} \Big)^{\frac{1}{m-1}}\\
&\rightarrow \max\{1,C_0\}~~\text{as}~~m\rightarrow \infty.
\end{aligned}
\end{equation*}
where the uniform constant $C_0>0$ is given in \eqref{a1}. This implies that for $m> m_0$ with some suitably large $m_0>2$,
\begin{equation*}
\begin{aligned}
\sup\limits_{t\in[0,T]}\|n_\varepsilon(t)\|_{L^{m-1}(\mathbb{R}^d)}\leq 2\max\{1,C_0\}.
\end{aligned}
\end{equation*}
On the other hand, one directly concludes from \eqref{a1} and \eqref{PL1var} that, for some constant $C(m_0,C_0)>0$,
\begin{equation*}
\begin{aligned}
\sup\limits_{t\in[0,T]}\|n_\varepsilon(t)\|_{L^{m-1}(\mathbb{R}^d)}\le C \Big((m-2)\|c_0\|_{L^2(\mathbb{R}^d)}^2+\|n_0\|_{L^{m-1}(\mathbb{R}^d)}^{m-1}\Big)^{\frac{1}{m-1}}
\le C(m_0,C_0),
\qquad 2<m\le m_0.
\end{aligned}
\end{equation*}
Consequently, we have \eqref{nm1var} for $p=m-1$.
This, combined with \eqref{211} and the interpolation between $L^1(\mathbb{R}^d)$ and $L^{m-1}(\mathbb{R}^d)$, implies  \eqref{nm1var} for $1< p<m-1$.
\end{proof}

\begin{lemma}\label{lemma23}
Let $m\geq 3$, and $T>0$ be any given time. Then, under the assumptions \eqref{ca} and \eqref{a1}, it holds that 
\begin{equation}\label{uL2var}
\begin{aligned}
&\sup_{t\in[0,T]}\|u_{\varepsilon}(t)\|_{L^2(\mathbb{R}^d)}^2+\|\nabla u_{\varepsilon}\|_{L^2(Q_T)}^2\leq Ce^{CT}\big(\|u_0\|_{L^2(\mathbb{R}^d)}^2+TN_0^2\big),
\end{aligned}
\end{equation}
and
\begin{equation}
\begin{aligned}
&\sup_{t\in[0,T]}\|\nabla c_{\varepsilon}(t)\|_{L^2(\mathbb{R}^d)}^2+\int_0^T \|\nabla^2 c_{\varepsilon}(t)\|_{L^2(\mathbb{R}^d)}^2\,dt\\
&\quad \quad\leq C\bigg(\|\nabla c_0\|_{L^2(\mathbb{R}^d)}^2+c_B^2\Big(e^{CT}\|u_0\|_{L^2(\mathbb{R}^d)}^2+TN_0^2\Big)\bigg),\label{nablacL2var}
\end{aligned}
\end{equation}
where $C>0$ is a constant independent of $T$ and $m$, and $N_0$ is given by \eqref{nm1var}.
\end{lemma}

\begin{proof}
Testing $\eqref{app}_3$ by $u_\var$ and employing \eqref{duniform}, the Cauchy-Schwarz inequality  and the condition \eqref{ca} on $\phi$, we have
\begin{equation*}
\begin{aligned}
\frac{1}{2}\|u_\varepsilon  \|_{L^2(\mathbb{R}^d)}^2+\int_0^t\|\nabla u_\varepsilon  \|_{L^2(\mathbb{R}^d)}^2\,d\tau
=&\frac{1}{2}\|u_{0,\varepsilon} \|_{L^2(\mathbb{R}^d)}^2+\iint_{Q_t}n_\varepsilon\nabla \phi \cdot u_\varepsilon  \,dxd\tau\\
\leq& \frac{1}{2}\|u_{0} \|_{L^2(\mathbb{R}^d)}^2+(1+\|\nabla \phi\|_{L^\infty(\mathbb{R}^d)})\int_0^t(\|n_\varepsilon  \|_{L^2(\mathbb{R}^d)}^2+\|u_\varepsilon  \|_{L^2(\mathbb{R}^d)}^2)\,d\tau.
\end{aligned}
\end{equation*}
Note that $\|n_\varepsilon(t) \|_{L^2(\mathbb{R}^d)}^2$ is uniformly bounded in time due to \eqref{nm1var} and $m\geq3$. Using the  Gr\"{o}nwall inequality, one gets \eqref{uL2var} immediately.

Next, we perform the $L^2$-estimate of $\nabla c_\var$. 
Testing $\eqref{app}_2$ by $-\Delta c_\var$, we have
\begin{equation}\label{um1var}
\begin{aligned}
\frac{1}{2}\|\nabla& c_\varepsilon  \|_{L^2(\mathbb{R}^d)}^2+\int_0^t\|\Delta c_\varepsilon  \|_{L^2(\mathbb{R}^d)}^2\,d\tau\\
&=\frac{1}{2}\|\nabla c_{0,\varepsilon} \|_{L^2(\mathbb{R}^d)}^2-\iint_{Q_t} u_\varepsilon  \cdot\nabla c_\varepsilon\Delta c_\varepsilon  \,dxd\tau-\iint_{Q_t}n_\varepsilon  f(c_\varepsilon) \Delta c_\varepsilon  \,dxd\tau\\
&\leq \frac{1}{2}\|\nabla c_{0} \|_{L^2(\mathbb{R}^d)}^2+\frac{1}{4}\int_0^t\|\Delta c_\varepsilon  \|_{L^2(\mathbb{R}^d)}^2\,d\tau\\
&\quad+\iint_{Q_t}n_\var^2f^2(c_\varepsilon)\,dxd\tau-\iint_{Q_t}  u_\varepsilon  \cdot\nabla c_\varepsilon\Delta c_\varepsilon  \,dxd\tau.
\end{aligned}
\end{equation}
For the last term on the right-hand side of \eqref{um1var}, integrating by parts and using \eqref{214}, Young's inequality and $\|\nabla^2 c_\varepsilon\|_{L^2(\mathbb{R}^d)}\leq C\|\Delta c_\varepsilon\|_{L^2(\mathbb{R}^d)}$, we can directly calculate that
\begin{equation}\label{um2var}
    \begin{aligned}
-\iint_{Q_t}  u_\varepsilon  \cdot\nabla c_\varepsilon\Delta c_\varepsilon  \,dxd\tau=&-\sum_{i,j}\iint_{Q_t} u_\varepsilon^i \partial_i c_\varepsilon  \partial_{jj} c_\varepsilon  \,dxd\tau\\
=&\underbrace{\sum_{i,j}\iint_{Q_t} u_\varepsilon^i\partial_{ij} c_\varepsilon  \partial_{j} c_\varepsilon  \,dxd\tau}_{=0}+\sum_{i,j}\iint_{Q_t} \partial_{j}u_\varepsilon^i \partial_{i} c_\varepsilon  \partial_{j} c_\varepsilon  \,dxd\tau\\
=&\underbrace{-\sum_{i,j}\iint_{Q_t} \partial_{ij}  u_\varepsilon^i c_\varepsilon  \partial_{j} c_\varepsilon  \,dxd\tau}_{=0}-\sum_{i,j}\iint_{Q_t} \partial_{j} u_\varepsilon^i c_\varepsilon  \partial_{ij} c_\varepsilon  \,dxd\tau\\
\leq & Cc_B^2\int_0^t \|\nabla u_\varepsilon  \|_{L^2(\mathbb{R}^d)}^2\,d\tau+\frac{1}{4}\int_0^t\|\Delta c_\varepsilon  \|_{L^2(\mathbb{R}^d)}^2\,d\tau.
    \end{aligned}
\end{equation}
The combination of \eqref{um1var} and \eqref{um2var} gives rise to
\begin{equation}\nonumber
    \begin{aligned}
        &\|\nabla c_\varepsilon  \|_{L^2(\mathbb{R}^d)}^2+\int_0^t\|\Delta c_\varepsilon  \|_{L^2(\mathbb{R}^d)}^2\,d\tau\leq\|\nabla c_0 \|_{L^2(\mathbb{R}^d)}^2+ 2 c_B^2\int_0^t\|\nabla u_\varepsilon  \|_{L^2(\mathbb{R}^d)}^2\,d\tau+C\int_0^t\|n_\varepsilon  \|_{L^2(\mathbb{R}^d)}^2\,d\tau.
    \end{aligned}
\end{equation}
Together with \eqref{nm1var}, \eqref{uL2var} and $\|\nabla^2 c_\varepsilon\|_{L^2(\mathbb{R}^d)}\leq C\|\Delta c_\varepsilon\|_{L^2(\mathbb{R}^d)}$, we conclude \eqref{nablacL2var}.
\end{proof}

Finally, we establish uniform estimates of the time derivatives of $(n_\varepsilon,c_\varepsilon,u_\varepsilon)$, which play an indispensable role in proving the strong convergence of $(n_{\varepsilon},c_{\varepsilon},u_{\varepsilon})$.

\begin{lemma}
Let $m\geq 3$ and $s_0>\frac{d}{2}$. Set
$p_m=\frac{(d+2)(m-1)}{d}$ if $d\geq3$ and $p_m=2(m-1)$ if $d=2$.
It holds true that
\begin{equation}\label{n1var}
\begin{aligned}
&\|\partial_t n_\varepsilon\|_{L^{p_m}(0,T;H^{-s_0-2}(\mathbb{R}^d))}\leq C_{m,T},
\end{aligned}
\end{equation}
and
\begin{equation}
\begin{aligned}
&\|\partial_t c_\varepsilon\|_{L^2(0,T;H^{-1}(\mathbb{R}^d))}
+\|\partial_t u_\varepsilon\|_{L^{2}(0,T;H^{-s_0-1}(\mathbb{R}^d))}
\leq C_{T},\label{cutvar}
\end{aligned}
\end{equation}
where $C_{m,T}>0$ is independent of $\varepsilon$, and $C_T>0$ is independent of both $m$ and $\varepsilon$.
\end{lemma}

\begin{proof}
Before analyzing $\partial_t n_\var$, we estimate $n_\var^m$. In the case $d\geq3$, since $\frac{2d}{d-2}>\frac{m}{m-1}>1$, it follows from H\"older's and Sobolev's inequalities that
\begin{equation*}
\begin{aligned}
\|n^m_\varepsilon\|_{L^{1}(\mathbb{R}^d)}&=\|n_\varepsilon^{m-1}\|_{L^{\frac{m}{m-1}}(\mathbb{R}^d)}^{\frac{m}{m-1}}\\
&\leq \|n_\varepsilon^{m-1}\|_{L^1(\mathbb{R}^d)}^{(1-\frac{2d}{(d+2)m})\frac{m}{m-1}} \|n_\varepsilon^{m-1}\|_{L^{\frac{2d}{d-2}}(\mathbb{R}^d)}^{\frac{2d}{(d+2)(m-1)}}\\
&\leq C \|n_\varepsilon^{m-1}\|_{L^1(\mathbb{R}^d)}^{(1-\frac{2d}{(d+2)m})\frac{m}{m-1}}  \|\nabla n^{m-1}_\varepsilon\|_{L^2(\mathbb{R}^d)}^{\frac{2d}{(d+2)(m-1)}}.
\end{aligned}
\end{equation*}
In the case $d=2$, applying the Gagliardo-Nirenberg--Sobolev inequality to $g=n_\varepsilon^{m-1}$ with $p=\frac{m}{m-1}$, $q=1$, $r=2$ and $\alpha=\frac1m$, we derive
\begin{equation*}
\begin{aligned}
\|n^m_\varepsilon\|_{L^{1}(\mathbb{R}^d)}=\|n_\varepsilon^{m-1}\|_{L^{\frac{m}{m-1}}(\mathbb{R}^d)}^{\frac{m}{m-1}}\leq C\|n^{m-1}_\varepsilon\|_{L^1(\mathbb{R}^d)}\|\nabla n^{m-1}_\varepsilon\|_{L^2(\mathbb{R}^d)}^{\frac{1}{m-1}}.
\end{aligned}
\end{equation*}
As $m\geq 2$, combining the above two cases with \eqref{PL1var}, we know that
\begin{align}
\|n_\var^m\|_{L^{p_m}(0,T;L^{1}(\mathbb{R}^d))}\leq C_{m,T}.\label{nmLp}
\end{align}
Note that 
$$
\partial_t n_\varepsilon=-\nabla\cdot \big( u_\varepsilon  n_\varepsilon\big)+\Delta n_\var^m+\varepsilon\Delta n_\varepsilon-\nabla\cdot \big(\chi(c_\varepsilon) n_\varepsilon  \nabla(J_\varepsilon\ast c_\varepsilon)\big).
$$
For any $\varphi\in L^{\frac{p_m}{p_m-1}}(0,T;H^{s_0+2}(\mathbb{R}^d))$ with $p_m$ defined above, we deduce from \eqref{nm1var}, \eqref{uL2var}, \eqref{nablacL2var}, \eqref{nmLp} and the embedding property $H^{s_0+2}(\mathbb{R}^d)\hookrightarrow W^{2,\infty}(\mathbb{R}^d)$ that
\begin{equation*}
\begin{aligned}
&\quad \Big|\iint_{Q_T} \partial_t n_\varepsilon   \varphi  \,dxdt \Big|\\
\leq& \|n_\var^m\|_{L^{p_m}(0,T;L^{1}(\mathbb{R}^d))} \|\Delta \varphi\|_{L^{\frac{p_m}{p_m-1}}(0,T;L^{\infty}(\mathbb{R}^d))}+\varepsilon  \|n_\varepsilon\|_{L^{\infty}(0,T;L^1(\mathbb{R}^d))}\|\Delta \varphi\|_{L^{1}(0,T;L^{\infty}(\mathbb{R}^d))}\\
&+\|u_\varepsilon\|_{L^{\infty}(0,T;L^2(\mathbb{R}^d))}\|n_\varepsilon\|_{L^{\infty}(0,T;L^2(\mathbb{R}^d))}\|\nabla\varphi\|_{L^1(0,T;L^{\infty}(\mathbb{R}^d))}\\
&+\sup_{0\leq s\leq c_B}|\chi(s)|\|n_\varepsilon\|_{L^{\infty}(0,T;L^2(\mathbb{R}^d))}\|\nabla c_\varepsilon\|_{L^{\infty}(0,T;L^2(\mathbb{R}^d))} \|\nabla\varphi\|_{L^1(0,T;L^{\infty}(\mathbb{R}^d))}\leq C_{m,T}.
\end{aligned}
\end{equation*}
This implies \eqref{n1var}.

On the other hand, note $\partial_t c_\varepsilon=\Delta c_\varepsilon-n_\varepsilon  f(c_\varepsilon)-\nabla\cdot\big( u_\varepsilon  c_\varepsilon\big)$. It is clear that, due to \eqref{214}, \eqref{nm1var}, \eqref{uL2var} and \eqref{nablacL2var}, $\Delta c_\var$ and $n_\varepsilon  f(c_\varepsilon)$ are uniformly bounded in $L^2(0,T;L^2(\mathbb{R}^d))$, and  we have
\begin{equation}
\begin{aligned}
\big\|\nabla\cdot\big( u_\varepsilon  c_\varepsilon\big)\big\|_{L^{2}(0,T;H^{-1}(\mathbb{R}^d))}&\leq CT^{1/2}\| u_\varepsilon  c_\varepsilon  \|_{L^{\infty}(0,T;L^2(\mathbb{R}^d))}\leq  CT^{1/2}\|c_\varepsilon\|_{L^{\infty}(Q_T)}\|u_\varepsilon\|_{L^{\infty}(0,T;L^{2}(\mathbb{R}^d))}\leq C_T.\nonumber
\end{aligned}
\end{equation}
Thus, we conclude the uniform $L^2(0,T;H^{-1}(\mathbb{R}^d))$ bound of $\partial_t c_\var$ in \eqref{cutvar}.

Furthermore,  we can write
$\partial_t u_\varepsilon=-\mathbf{P}\nabla\cdot\big( u_\varepsilon\otimes u_\varepsilon\big)+\Delta u_\varepsilon- \mathbf{P}(n_\varepsilon  \nabla\phi)$, where the incompressible projection $\mathbf{P}$ is defined by $\mathbf{P}:={\rm Id}+\nabla(-\Delta)^{-1}\nabla\cdot$. We know that $\Delta u_\var$ and $\mathbf{P}(n_\varepsilon  \nabla\phi)$ are, respectively, uniformly bounded in $L^2(0,T;\dot{H}^{-1}(\mathbb{R}^d))$ and $L^{\infty}(0,T;L^2(\mathbb{R}^d))$. For any $\psi\in L^{2}(0,T;H^{s_0+1}(\mathbb{R}^d))$, in view of \eqref{uL2var} and Sobolev's embeddings, we get
\begin{equation*}
\begin{aligned}
\Big|\iint_{Q_T} \mathbf{P}\nabla\cdot \big( u_\varepsilon\otimes u_\varepsilon) \psi \,dx dt\Big|
&\leq CT^{\frac{1}{2}}\|u_\varepsilon\|_{L^{\infty}(0,T;L^2(\mathbb{R}^d))}^2 \|\nabla\psi\|_{L^{2}(0,T;H^{s_0}(\mathbb{R}^d))}\leq CT.
\end{aligned}
\end{equation*}
Hence, we justify the uniform bound of $\partial_t u_\var$ in \eqref{cutvar}.
\end{proof}

\subsection{Proof of Theorem \ref{thm1}}

We are in a position to prove the convergence of the approximate sequence $\{(n_\varepsilon,c_\varepsilon,u_\varepsilon)\}_{0<\varepsilon<1}$. The uniform bounds in Lemmas \ref{lemma21}-\ref{lemma23} ensure that as $\varepsilon\rightarrow 0$, 
there exists a limit $(n,c,u)$ such that, for $m\geq3$, up to subsequences,
\begin{equation}\label{weak}
\left\{
\begin{aligned}
n_\varepsilon  \rightharpoonup n \quad & \text{weakly$^*$} && \text{in } L^{\infty}(0,T;L^{p}(\mathbb{R}^d)),\qquad 1\leq p\leq m-1,\\
c_\varepsilon  \rightharpoonup c \quad & \text{weakly$^*$} && \text{in } L^{\infty}(0,T;L^1(\mathbb{R}^d)\cap L^{\infty}(\mathbb{R}^d)\cap H^1(\mathbb{R}^d)),\\
c_\varepsilon  \rightharpoonup c \quad & \text{weakly}   && \text{in } L^2(0,T;H^2(\mathbb{R}^d)),\\
u_\varepsilon  \rightharpoonup u \quad & \text{weakly$^*$} && \text{in } L^{\infty}(0,T;L^2(\mathbb{R}^d)),\\
u_\varepsilon  \rightharpoonup u \quad & \text{weakly}   && \text{in } L^2(0,T;H^1(\mathbb{R}^d)).
\end{aligned}
\right.
\end{equation}
and for any function $\varphi\in \mathcal{C}_0^\infty(Q_T)$,  
\begin{equation}\label{strong0}
\begin{aligned}
\varepsilon  \Big|\int_{Q_T} n_\varepsilon  \Delta \varphi \,dx dt\Big|&\leq C\varepsilon  \|n_\varepsilon\|_{L^{\infty}(0,T;L^1(\mathbb{R}^d))} \|\Delta \varphi\|_{L^1(0,T;L^{\infty}(\mathbb{R}^d))}\rightarrow 0\text{ as }\varepsilon\to0.
\end{aligned}
\end{equation}

In order to prove the convergence of all nonlinear terms in \eqref{app}, we need the strong convergence of $(n_\varepsilon,c_\varepsilon,u_\varepsilon)$ in a suitable sense. Recall the uniform bounds \eqref{213}, \eqref{uL2var}, \eqref{nablacL2var} for $c_\varepsilon, u_\var$ and \eqref{cutvar} for $\partial_t c_\varepsilon, \partial_{t} u_\var$. Thus, applying the Aubin-Lions-Simon lemma (Lemma \ref{lemmaAubin}) and the Cantor diagonal argument, we prove that, up to a subsequence, as $\varepsilon\rightarrow 0$,
\begin{equation}\label{strong1}
\left\{
\begin{alignedat}{3}
c_\varepsilon  \to c \quad
&\text{strongly}\quad &\text{in}\quad &L^2(0,T;H^1_{\mathrm{loc}}(\mathbb{R}^d)),\\
u_\varepsilon  \to u \quad
&\text{strongly}\quad &\text{in}\quad &L^2(0,T;L^2_{\mathrm{loc}}(\mathbb{R}^d)).
\end{alignedat}
\right.
\end{equation}
From this and the fact that $c_\var$ is uniformly bounded in $L^{\infty}(Q_T)$, we infer that, as $\varepsilon\rightarrow0$, for any $1\leq p<\infty$,
\begin{equation}\label{cstrong}
\left\{
\begin{alignedat}{3}
\chi(c_\varepsilon)\to \chi(c)\quad
&\text{strongly}\quad &\text{in}\quad &L^p_{\mathrm{loc}}(Q_T),\\
f(c_\varepsilon)\to f(c)\quad
&\text{strongly}\quad &\text{in}\quad &L^p_{\mathrm{loc}}(Q_T).
\end{alignedat}
\right.
\end{equation}

The nonlinear diffusion term requires a strong convergence of $n_\var$. To achieve it, one deduces from \eqref{PL1var} and \eqref{n1var}, Lemma \ref{lemmaDub} (Dubinski\"i compactness lemma) and the Cantor diagonal argument that, up to a subsequence, as $\varepsilon\rightarrow 0$,
\begin{equation}
\begin{aligned}
n_\varepsilon\rightarrow n\quad& \text{strongly~~in ~~} L^2_{{\rm{loc}}}(Q_T).\label{strong2}
\end{aligned}
\end{equation}
We now claim that there exists some $r_1, r_2>1$ and a constant $C_{m,T}$ independent of $\var$ such that
\begin{align}
\|n^m_\varepsilon\|_{L^{r_1}(0,T;L^{r_2})}\leq C_{m,T}.\label{323}
\end{align}
Indeed, in the case $d\geq3$, we set $r_1=\frac{2(m-1)}{m}>1$ and $r_2=\frac{2d}{d-2}\frac{m-1}{m}>1$ with $m\geq3$. By virtue of Sobolev's inequality, $\frac{2d(m-1)}{d-2}>m$, we have 
$$
\int_0^T \|n^m_\varepsilon\|_{L^{r_2}(\mathbb{R}^d)}^{r_1}dt\leq \int_0^T \|n^{m-1}_{\varepsilon}\|_{L^{\frac{2d}{d-2}}(\mathbb{R}^d)}^{\frac{2d}{(d-2)r_2}r_1}dt\leq \int_0^T \|\nabla n^{m-1}_{\varepsilon}\|_{L^2(\mathbb{R}^d)}^{2}dt\leq C_{m,T}.
$$
When $d=2$, we let $r_1=\frac{2(m-1)}{m(1-\beta)}>1$ and $r_2=2$ with $\beta=\frac{m-1}{2m}\in(0,1)$. Then, the Gagliardo-Nirenberg-Sobolev inequality (Lemma~\ref{t9}) guarantees that
\begin{align*}
\int_0^T \|n^m_\varepsilon\|_{L^{r_2}(\mathbb{R}^d)}^{r_1}dt&=\int_0^T \|n^{m-1}_{\varepsilon}\|_{L^{\frac{2m}{m-1}}(\mathbb{R}^d)}^{r_1\frac{m}{m-1}}dt\\
&\leq \int_0^T \|n^{m-1}_{\varepsilon}\|_{L^{1}(\mathbb{R}^d)}^{r_1\frac{m}{m-1}\beta}\|\nabla n^{m-1}_{\varepsilon}\|_{L^{2}(\mathbb{R}^d)}^{r_1\frac{m}{m-1}(1-\beta)}dt\\
&\leq \|n^{m-1}_{\varepsilon}\|_{L^{\infty}(0,T;L^{1}(\mathbb{R}^d))}^{r_1\frac{m}{m-1}\beta}\int_0^T \|\nabla n^{m-1}_{\varepsilon}\|_{L^{2}(\mathbb{R}^d)}^{2}dt\leq C_{m,T}.
\end{align*}
Thus, \eqref{323} is proved. Since $n_\var$ converges to $n$ a.e. on any compact subset of $[0,T)\times\mathbb{R}^d$ due to \eqref{strong2}, it holds by \eqref{323} that
\begin{equation}\label{nmstrong}
\begin{aligned}
n^{m}_\varepsilon\rightarrow  n^m \quad& \text{strongly~~in ~~} L^1_{{\rm{loc}}}(Q_T)\text{ as }\varepsilon\to0.
\end{aligned}
\end{equation}
Finally, let $p_c>2$ be given by $p_c=4$ for $d=2$ and $p_c=\frac{2d}{d-2}$ for $d\geq3$. Let $\varphi$ be any smooth function supported in $[0,T)\times K$, where $K$ is a compact subset of $\mathbb{R}^d$. Lemmas \ref{lemma21} and \ref{lemma23} imply that $\nabla c_\var$ is uniformly (in $\var$) bounded in $L^2(0,T;L^2(\mathbb{R}^d)\cap L^{p_c}(\mathbb{R}^d))$. Consequently, we have
\begin{equation}\label{nonlinearstrong}
\begin{aligned}
&\Big|\iint_{Q_T}n_\varepsilon\chi(c_\varepsilon) \nabla( J_{\varepsilon}\ast c_\varepsilon)\cdot \nabla\varphi \,dxdt-\iint_{Q_T}\chi(c)n \nabla c\cdot \nabla\varphi \,dx dt\Big|\\
&\quad\leq \|\chi(c_\varepsilon)- \chi(c)\|_{L^{\frac{2 p_c}{p_c-2}}(0,T;L^{\frac{2 p_c}{p_c-2}}(K))} \|n_\varepsilon\|_{L^{\infty}(0,T;L^2(\mathbb{R}^d))} \|\nabla c_\varepsilon\|_{L^2(0,T;L^{p_c}(\mathbb{R}^d))}\|\nabla\varphi\|_{L^{\infty}(Q_T)}\\
&\quad\quad+\|\chi(c)\|_{L^{\infty}(Q_T)} \|n_\varepsilon- n\|_{L^2(0,T;L^2(K))} \|\nabla c_\varepsilon\|_{L^2(Q_T)} \|\nabla\varphi\|_{L^{\infty}(Q_T)}\\
&\quad\quad+\|\chi(c)\|_{L^{\infty}(Q_T)} \|n\|_{L^{2}(Q_T)} \|\nabla c_\varepsilon  -\nabla c\|_{L^2(0,T;L^2(K))} \|\nabla\varphi\|_{L^{\infty}(Q_T)}\\
&\quad\quad+\|\chi(c)\|_{L^{\infty}(Q_T)} \|n\|_{L^{2}(Q_T)} \|\nabla( J_{\varepsilon}\ast c)-\nabla c\|_{L^2(Q_T)}\|\nabla\varphi\|_{L^{\infty}(Q_T)}\\
&\quad \rightarrow 0\text{ as }\varepsilon\to0.
\end{aligned}
\end{equation}
Combining the convergence results  \eqref{weak}-\eqref{nonlinearstrong}, we prove that in the sense of distributions, the approximate equations \eqref{app} converge to the original equations \eqref{eq1-1} as $\varepsilon\rightarrow 0$, and the limit $(n, c, u)$ is a global  weak solution to the Cauchy problem \eqref{eq1-1}-\eqref{d} in the sense of Definition \ref{ws}.  Indeed, with the uniform estimates \eqref{211}-\eqref{214}, \eqref{PL1var}-\eqref{nm1var} and \eqref{uL2var}-\eqref{nablacL2var} and Fatou's property, the limit $(n, c, u)$ has the regularity properties \eqref{r1} and \eqref{basep}. 
\hfill $\Box$


\section{Hele-Shaw limit} \label{sec:hs}
We first establish some uniform-in-$m$ regularity estimates for the solution $(n_m,c_m,u_m,P_m,\Pi_m)$ of the Cauchy problem \eqref{eqm}. Then, based on these uniform bounds, we verify some desired weak and strong convergence properties and justify the validity of the Hele-Shaw limit as $m\rightarrow \infty$. In addition, we choose a suitable test function acting on the resulting Hele-Shaw type system, which leads to the complementarity relation. 


\subsection{Uniform estimates}

The key uniform-in-$m$ regularity estimates are given as follows. 

\begin{prop}[Uniform regularity estimates]\label{propuee} Let $d\geq 2$, $m\geq\max\{d+1,5\}$. Let the assumptions \eqref{ca}, \eqref{a1} and \eqref{H3} on $f, \chi, \phi$ and the initial data $(n_{m,0},c_{m,0},u_{m,0})$ hold, and let $(n_m,c_m,u_m,P_m,\Pi_m)$ be the global weak solution to the Cauchy problem \eqref{eqm} obtained in Theorem \ref{thm1}. Let $T>0$ be any given time. Then, for $t\in[0,T]$, $(n_m,c_m,u_m,P_m,\Pi_m)(t)$ fulfills the following properties:
\begin{itemize}
\item  {\rm(}Uniform estimates of $n_m$ and $P_m${\rm)}: 
\[
\begin{array}{@{}l@{\qquad\qquad}l@{}}
\|n_m(t)\|_{L^1(\mathbb{R}^d)\cap L^{m+1}(\mathbb{R}^d)}\le C_T,
&
\|n_m^m\|_{L^2(0,T;H^1(\mathbb{R}^d))}\le C_T,\\[0.3em]
m\|(n_m-1)_+\|_{L^2(Q_T)}^2\le C_T,
&
\|P_m\|_{L^2(0,T;H^1(\mathbb{R}^d))}\le C_T.
\end{array}
\]
\item  {\rm(}Uniform estimates of $c_m${\rm)}:
\[
\begin{array}{@{}l@{\qquad\qquad}l@{}}
\|c_m(t)\|_{L^1(\mathbb{R}^d)}\le \|c_{m,0}\|_{L^1(\mathbb{R}^d)},
&
0\le c_m\le c_B,
\\[0.3em]
\|c_m(t)\|_{H^1(\mathbb{R}^d)}\le C_T,
&
\|\nabla c_m\|_{L^2(0,T;H^1(\mathbb{R}^d))}\le C_T.
\end{array}
\]
\item {\rm(}Uniform estimates of $u_m$ and $\Pi_m${\rm)}:
\[
\begin{array}{@{}l@{\qquad\qquad}l@{}}
\|u_m(t)\|_{L^2(\mathbb{R}^d)}\le C_T,
&
\|\nabla u_m\|_{L^2(Q_T)}\le C_T,
\\[0.3em]
\|\Pi_m\|_{L^{\frac{2(d-1)}{d}}(0,T;L^{\frac{d-1}{d-2}}(\mathbb{R}^d))}\le C_T,
&
d\ge 3,
\\[0.3em]
\|\Pi_m\|_{L^{\frac32}(0,T;L^3(\mathbb{R}^2))}\le C_T,
&
d=2.
\end{array}
\]
\item {\rm(}Uniform estimates of time derivatives{\rm)}:
\begin{align*}
&\|\partial_t n_m\|_{L^2(0,T;\dot{H}^{-1}(\mathbb{R}^d))}\leq C_T,&&\\
&\|\partial_tc_m\|_{L^2(0,T;\dot{H}^{-1}(\mathbb{R}^d))}+\|\partial_t u_m\|_{L^{\frac{2(d-1)}{d}}(0,T;\dot{W}^{-1,\frac{d-1}{d-2}}(\mathbb{R}^d))}\leq C_T,&&d\geq 3,\\
&\|\partial_tc_m\|_{L^2(Q_T)}+\|\partial_t u_m\|_{L^{\frac{3}{2}}(0,T;\dot{W}^{-1,3}(\mathbb{R}^2))}\leq C_T, &&d=2.
\end{align*}
Here $C_T>0$ is a constant independent of $m$ but depending on $T$.
\end{itemize}
\end{prop}

We split the proof of Proposition \ref{propuee} into Lemmas \ref{lemma41}-\ref{lemma44} below.

First, due to the uniform bounds obtained in Lemmas \ref{lemma21}-\ref{lemma23} and Fatou's property, we have some basic uniform estimates associated with \eqref{r1}.

\begin{lemma}\label{lemma41}
Under the assumptions of Proposition \ref{propuee}, it holds for any $m\geq3$ and $t\in[0,T]$ that
\begin{align}
&\|n_m(t)\|_{L^{1}(\mathbb{R}^d)\cap L^{m-1}(\mathbb{R}^d)}\le C_T,
&& \label{basic1}\\
&\|P_m(t)\|_{L^1(\mathbb{R}^d)}\le C(m-2),
&&\|\nabla P_m\|_{L^2(0,T;L^2(\mathbb{R}^d))}\le C_T,\label{basic2}\\
&\|c_m(t)\|_{L^1(\mathbb{R}^d)}\le \|c_{m,0}\|_{L^1(\mathbb{R}^d)},
&&0\le c_m\le c_B,\label{basic3}\\
&\|c_m(t)\|_{H^1(\mathbb{R}^d)}\le C_T,
&&\|c_m\|_{L^2(0,T;H^2(\mathbb{R}^d))}\le C_T,\label{basic4}\\
&\|u_m(t)\|_{L^{2}(\mathbb{R}^d)}\le C_T,
&&\|u_m\|_{L^2(0,T;H^1(\mathbb{R}^d))}\le C_T.\label{basic5}
\end{align}
\end{lemma}

Based on Lemma \ref{lemma41}, we have the uniform estimates of the time derivatives $\partial_t c_m$ and $\partial_t u_m$ and the pressure $\Pi_m$.
\begin{lemma}
Under the assumptions of Proposition \ref{propuee}, we have 
\begin{align}
&\|\partial_t c_m\|_{L^2(0,T;\dot{H}^{-1}(\mathbb{R}^d))}\leq C_T,\label{time1}\\
&\|\Pi_m\|_{L^{\frac{2(d-1)}{d}}(0,T;L^{\frac{d-1}{d-2}}(\mathbb{R}^d))}\leq C_T,\label{time2}\\
&\|\partial_t u_m\|_{L^{\frac{2(d-1)}{d}}(0,T;\dot{W}^{-1,\frac{d-1}{d-2}}(\mathbb{R}^d))}\leq C_T,\label{time3}
\end{align}
for $d\geq3$ and
\begin{align}
&\|\partial_t c_m\|_{L^2(Q_T)}\leq C_T,\label{time4}\\
&\|\Pi_m\|_{L^{\frac{3}{2}}(0,T;L^{3}(\mathbb{R}^2))}\leq C_T,\label{time5}\\
&\|\partial_t u_m\|_{L^{\frac{3}{2}}(0,T;\dot{W}^{-1,3}(\mathbb{R}^2))}\leq C_T,\label{time6}
\end{align}
for $d=2$.
\end{lemma}
\begin{proof}
We first analyze the time derivative $\partial_t c_m$ like \eqref{time1} for $d\geq 3$ and \eqref{time4} for $d=2$. 
 For the $d$-dimensional case ($d\geq 3$), as $m\geq 3\geq \frac{2d}{d+2}+1$, it holds by the embedding $L^{\frac{2d}{d+2}}(\mathbb{R}^d)\hookrightarrow \dot{H}^{-1}(\mathbb{R}^d)$ that
\begin{equation*}
\begin{aligned}
\|n_mf(c_m)\|_{L^2(0,T;\dot{H}^{-1}(\mathbb{R}^d))}&\leq C\|n_mf(c_m)\|_{L^2(0,T;L^{\frac{2d}{d+2}}(\mathbb{R}^d))}\\
&\leq CT^{1/2} \|f(c_m)\|_{L^{\infty}(Q_T)}\|n_m\|_{L^{\infty}(0,T; L^{\frac{2d}{d+2}}(\mathbb{R}^d))}\\
&\leq C_T.
\end{aligned}
\end{equation*}
This, combined with  $\eqref{eqm}_2$ and \eqref{basic1}-\eqref{basic5}, yields
\begin{equation*}
\begin{aligned}
\|\partial_t c_m\|_{L^2(0,T;\dot{H}^{-1}(\mathbb{R}^d))}^2\leq &C\|\nabla c_m\|_{L^2(Q_T)}^2+C c_B^2\|u_m\|_{L^2(Q_T)}^2+C\|n_mf(c_m)\|_{L^2(0,T;\dot{H}^{-1}(\mathbb{R}^d))}^2\leq C_T,
\end{aligned}
\end{equation*}
for all $d\geq3$.

Thanks to  $\eqref{eqm}_2$ and \eqref{basic1}-\eqref{basic5},  the estimate \eqref{time4} is given by
\begin{equation*}
\begin{aligned}
\|\partial_t c_m\|_{L^2(Q_T)}\leq &\|\Delta c_m\|_{L^2(Q_T)}+\|n_m\|_{L^2(Q_T)}\|f(c_m)\|_{L^{\infty}(Q_T)}+\|u_m\cdot\nabla c_m\|_{L^2(Q_T)}\leq  C_T,\quad d=2,
\end{aligned}
\end{equation*}
where we used the Gagliardo-Nirenberg-Sobolev inequality (Lemma~\ref{t9}) such that
\begin{equation*}
    \begin{aligned}
  \|u_m\cdot\nabla c_m\|_{L^2(Q_T)}^2\leq& \int_0^T \|u_m\|_{L^4(\mathbb{R}^2)}^2\|\nabla c_m\|_{L^4(\mathbb{R}^2)}^2\,dt\\
  &\leq C\|u_m\|_{L^{\infty}(0,T;L^2(\mathbb{R}^2))}\|\nabla u_m\|_{L^2(Q_T)} \|\nabla c_m\|_{L^{\infty}(0,T;L^2(\mathbb{R}^2))} \|\nabla^2 c_m\|_{L^2(Q_T)} \leq C_T.
    \end{aligned}
\end{equation*}
 Consequently, \eqref{time1} and \eqref{time4} hold.

 Next, the step is to deal with \eqref{time2}-\eqref{time3} for $d\geq3$ and \eqref{time5}-\eqref{time6} for $d=2$.
Hence, by taking the divergence operator on $\eqref{eqm}_3$, we get 
\begin{equation*}
-\Delta \Pi_m
=
\nabla\cdot\nabla\cdot(u_m\otimes u_m)
+
\nabla\cdot(n_m\nabla\phi).
\end{equation*}
Due to the singular integral theory, $\nabla^2(-\Delta)^{-1}=(-\Delta)^{-1}\nabla^2$ maps $L^p(\mathbb{R}^d)$ to $L^p(\mathbb{R}^d)$ for any $1<p<\infty$.  Furthermore, for $d\geq3$, in view of Lemma~\ref{t9} and Sobolev's inequality,  we have
\begin{equation*}
\begin{aligned}
\|u_m^2&\|_{L^{\frac{d-1}{d-2}}(\mathbb{R}^d)}+\|\Pi_m\|_{L^{\frac{d-1}{d-2}}(\mathbb{R}^d)}\\
\leq& \|u_m^2\|_{L^{\frac{d-1}{d-2}}(\mathbb{R}^d)}+\|(-\Delta)^{-1}\nabla\cdot\nabla\cdot(u_m\otimes u_m)\|_{L^{\frac{d-1}{d-2}}(\mathbb{R}^d)}+\|(\Delta)^{-1}\nabla\cdot(n_m\nabla \phi)\|_{L^{\frac{d-1}{d-2}}(\mathbb{R}^d)}\\
\leq& C\|u_m^2\|_{L^{\frac{d-1}{d-2}}(\mathbb{R}^d)}+C\|n_m\nabla\phi \|_{L^{\frac{d^2-d}{d^2-d-1}}(\mathbb{R}^d)}\\
\leq & C\|u_m\|_{L^2(\mathbb{R}^d)}^{\frac{d-2}{d-1}}\|\nabla u_m\|_{L^2(\mathbb{R}^d)}^{\frac{d}{d-1}}+C\|n_m\|_{L^{\frac{d^2-d}{d^2-d-1}}(\mathbb{R}^d)}.
\end{aligned}
\end{equation*}
Since $
\frac{d^2-d}{d^2-d-1}\le 2\leq m-1$, it follows from  \eqref{basic1} and  \eqref{basic5} that
\begin{equation}\label{cp1}
     \begin{aligned}
\|u_m^2&\|_{L^{\frac{2(d-1)}{d}}(0,T;L^{\frac{d-1}{d-2}}(\mathbb{R}^d))}+\|\Pi_m\|_{L^{\frac{2(d-1)}{d}}(0,T;L^{\frac{d-1}{d-2}}(\mathbb{R}^d))}\\
\leq& C\|\nabla u_m\|_{L^2(Q_T)}^{\frac{d}{d-1}}+C\|n_m\|_{L^{\frac{2(d-1)}{d}}(0,T;L^{\frac{d^2-d}{d^2-d-1}}(\mathbb{R}^d))}\leq C_T.
     \end{aligned}
 \end{equation}
 Together with the equation of $u_m$, i.e., $\eqref{eqm}_3$, this gives rise to
\begin{equation}\label{ptdm}
\begin{aligned}
\|\partial_t u_m&\|_{L^{\frac{2(d-1)}{d}}(0,T;\dot{W}^{-1,\frac{d-1}{d-2}}(\mathbb{R}^d))}\\
\leq& C\| \Pi_m\|_{L^{\frac{2(d-1)}{d}}(0,T;L^{\frac{d-1}{d-2}}(\mathbb{R}^d))}+C\|u_m^2 \|_{L^{\frac{2(d-1)}{d}}(0,T;L^{\frac{d-1}{d-2}}(\mathbb{R}^d))}\\
&+C\|n_m\nabla\phi \|_{L^{\frac{2(d-1)}{d}}\big(0,T;L^{\frac{d^2-d}{d^2-d-1}}(\mathbb{R}^d)\big)}\leq C_T,
\end{aligned}
\end{equation}
where we used Sobolev's inequality.

The case $d=2$ can be handled in the same manner. Indeed, similar calculations lead to
\begin{equation*}
    \begin{aligned}
\|u_m^2&\|_{L^3(\mathbb{R}^2)}+\|\Pi_m\|_{L^3(\mathbb{R}^2)}\\
\leq&\|u_m^2\|_{L^3(\mathbb{R}^2)}+ \|(-\Delta)^{-1}\nabla\cdot\nabla\cdot(u_m\otimes u_m)\|_{L^3(\mathbb{R}^2)}+\|(-\Delta)^{-1}\nabla\cdot(n_m\nabla \phi)\|_{L^3(\mathbb{R}^2)}\\
\leq& C\|u_m^2\|_{L^3(\mathbb{R}^2)}+C\|n_m\nabla\phi \|_{L^{\frac{6}{5}}(\mathbb{R}^2)}\\
\leq & C\big(\|u_m\|_{L^2(\mathbb{R}^2)}\big)^{\frac{2}{3}}\big(\|\nabla u_m\|_{L^2(\mathbb{R}^2)}\big)^{\frac{4}{3}}+C\|n_m\|_{L^{\frac{6}{5}}(\mathbb{R}^2)},
    \end{aligned}
\end{equation*}
and
 \begin{equation}\label{cp22}
     \begin{aligned}
\|u_m^2&\|_{L^{\frac{3}{2}}(0,T;L^3(\mathbb{R}^2))}+\|\Pi_m\|_{L^{\frac{3}{2}}(0,T;L^3(\mathbb{R}^2))}\\
\leq& C_T\big(\|\nabla u_m\|_{L^2(Q_T)}\big)^{\frac{4}{3}}+C\|n_m\|_{L^{\frac{3}{2}}(0,T;L^{\frac{6}{5}}(\mathbb{R}^2))}\leq C_T.
     \end{aligned}
 \end{equation}
Hence, it follows that 
\begin{equation}\label{cp222}
    \begin{aligned}
\|\partial_t u_m\|_{L^{\frac{3}{2}}(0,T;\dot{W}^{-1,3}(\mathbb{R}^2))}\leq& C \|u_m^2\|_{L^{\frac{3}{2}}(0,T;L^3(\mathbb{R}^2))}+\|\Pi_m\|_{L^{\frac{3}{2}}(0,T;L^3(\mathbb{R}^2))}\\
&+C\|n_m\|_{L^{\frac{3}{2}}(0,T;L^{\frac{6}{5}}(\mathbb{R}^2))}\leq C_T.
    \end{aligned}
\end{equation}
Combining \eqref{cp1},\eqref{ptdm}, \eqref{cp22} and \eqref{cp222}, we immediately
 infer \eqref{time2}-\eqref{time3} for $d\geq3$ and \eqref{time5}-\eqref{time6} for $d=2$.
\end{proof}

However, the uniform estimates \eqref{basic1}-\eqref{basic5} are not sufficient to carry out a compactness process for justifying the singular limit as $m\rightarrow\infty$ due to the lack of uniform regularity estimates for $\partial_t n_m$ and higher integrability of $n_m$. To overcome this difficulty, we need to establish additional regularity estimates of $n_m$.

\begin{lemma}\label{lemma43}Under the assumptions of Proposition \ref{propuee}, it holds for $m\geq\max\{d+1,5\}$ and $t\in[0,T]$ that
\begin{align}
&\|n_m(t)\|_{L^{m+1}(\mathbb{R}^d)}+\|\nabla n_m^m\|_{L^2(Q_T)}\leq C_T,\label{high1}\\
& \|P_m\|_{L^2(Q_T)}+ \|n_m^m\|_{L^2(Q_T)}\leq C_T,\label{high4}\\
&m\|(n_m-1)_+\|_{L^2(Q_T)}^2\leq C_T.\label{high3}
\end{align}
\end{lemma}

\begin{proof}
To prove \eqref{high1}, multiplying $\eqref{eqm}_1$ by $n_m^m$ and integrating the resulting equation over $Q_T$, we obtain
\begin{equation*}
    \begin{aligned}
\frac{1}{m+1}\sup\limits_{0\leq t\leq T}&\|n_m(t)\|_{L^{m+1}(\mathbb{R}^d)}^{m+1}+\|\nabla n_m^m\|_{L^2(Q_T)}^2\\
&\leq 
\frac{1}{4}\|\nabla n_m^m\|_{L^2(Q_T)}^2+\|n_m\chi(c_m)\nabla c_m\|_{L^2(Q_T)}^2+\frac{1}{m+1}\|n_{m,0}\|_{L^{m+1}(\mathbb{R}^d)}^{m+1},
    \end{aligned}
\end{equation*}
where the Cauchy-Schwarz inequality was used.
As $m-1\geq d$, one deduces from \eqref{basic1}, \eqref{basic4} and Sobolev's inequality that
\begin{equation*}
    \begin{aligned}
\|n_m\chi(c_m)\nabla c_m\|_{L^2(Q_T)}^2\leq& C\iint_{Q_T}n_m^2|\nabla c_m|^2  dxdt\\
\leq& C\int_{0}^T\hskip-4pt\Big(\int_{\mathbb{R}^d}n_m^ddx\Big)^{\frac{2}{d}}\Big(\int_{\mathbb{R}^d}|\nabla c_m|^{\frac{2d}{d-2}}\,dx\Big)^{\frac{d-2}{d}}\,dt\\
\leq&C\iint_{Q_T}|\nabla^2 c_m|^2  dxdt\leq C_T\text{ for }d\geq3.
    \end{aligned}
\end{equation*}
In the case $d=2$ and $m\geq 5$, we use the uniform bound of $n_m$ in
$L^\infty(0,T;L^1(\mathbb{R}^2)\cap L^{m-1}(\mathbb{R}^2))$ to obtain
\begin{equation*}
    \begin{aligned}
\|n_m\chi(c_m)\nabla c_m\|_{L^2(Q_T)}^2
\leq& C\int_{0}^T\hskip-4pt\Big(\int_{\mathbb{R}^2}n_m^4dx\Big)^{\frac{1}{2}}\Big(\int_{\mathbb{R}^2}|\nabla c_m|^{4}\,dx\Big)^{\frac{1}{2}}\,dt\\
\leq&C_T\int_{0}^T\hskip-4pt\Big(\int_{\mathbb{R}^2}|\nabla c_m|^2\,dx\Big)^{\frac{1}{2}}\Big(\int_{\mathbb{R}^2}|\nabla^2 c_m|^2\,dx\Big)^{\frac{1}{2}}\,dt\leq C_T.
    \end{aligned}
\end{equation*} 
Hence, we end up with \eqref{high1}.

 We turn to verify \eqref{high4}.  In the case $d\geq3$, it holds by \eqref{basic2}, \eqref{high1} and Sobolev's inequality that 
\begin{equation*}
\begin{aligned}
\|n_m^{m}\|_{L^2(Q_T)}^2&\leq C\iint_{Q_T}n_m^{2}P_m^2  dxdt\\
&\leq C \int_0^T\hskip-4pt(\int_{\mathbb{R}^d}n_m^ddx\big)^{\frac{2}{d}}\big(\int_{\mathbb{R}^d}P_m^{\frac{2d}{d-2}}\,dx\big)^{\frac{d-2}{d}}\,dt\\
&\leq C\iint_{Q_T}|\nabla P_m|^2  dxdt\leq C_T,\quad m+1\geq d.
\end{aligned}
\end{equation*}
By means of Young's inequality, we further obtain 
\begin{equation*}
\begin{aligned}
\iint_{Q_T}P_m^{2}  dxdt&\leq C\iint_{Q_T}n_m^{2m-2}  dxdt\\
&=C\iint_{Q_T}n_m^{\frac{2}{2m-1}}n_m^{2m\frac{2m-3}{2m-1}}  dxdt\\
&\leq \frac{2}{2m-1}C\iint_{Q_T}n_m  \,dxdt+ \frac{2m-3}{2m-1}C\iint_{Q_T}n_m^{2m}  dxdt\leq C_T,\quad d\geq3.
\end{aligned}
\end{equation*}
Consequently, we have \eqref{high4} for $d\geq3$.

 As for the two-dimensional case, the $L^p(\mathbb{R}^2)$-norm will not be achieved from Sobolev's embedding associated with $\dot{H}^1(\mathbb{R}^2)$. To overcome this difficulty, we observe that $\eqref{211}$ implies for sufficiently large $R$ that  
\[|\{n_m\geq 1\}\cap B_R|\leq \int_{B_R}n_m\,dx\leq \int_{\mathbb{R}^2}n_m\,dx\leq \int_{\mathbb{R}^2}n_{m,0}\,dx=:M,\]
which implies 
\[|\{n_m<1\}\cap B_R|=|B_R|-|\{n_m\geq 1\}\cap B_R|\geq |B_R|-M\geq \pi R^2-M \sim R^2 ~~\text{ for}\quad R\gg 1.\]
It is obvious that
\[(n_m^{m-1}-1)_+=0\quad\text{in }S:=\{n_m<1\}\cap B_R.\]
Furthermore, using Lemma~\ref{ei} with $S$, we deduce for $R\gg1$ that 
\begin{equation*}
    \begin{aligned}
\|(n_m^{m-1}-1)_+\|_{L^2(B_R)}&\leq \frac{CR}{|S|}\|\nabla(n_m^{m-1}-1)_+\|_{L^1(B_R)}\leq \frac{CR}{|S|}\|\nabla n_m^{m-1}\|_{L^1(B_R)}\\
&\leq \frac{C}{R}\|\nabla n_m^{m-1}\|_{L^2(B_R)}|B_R|^{\frac{1}{2}}\leq C\|\nabla n_m^{m-1}\|_{L^2(\mathbb{R}^2)}. 
    \end{aligned}
\end{equation*}
Letting $R\to\infty$, we take the $L^2$ integral in time and use \eqref{basic2} to obtain
\begin{equation*}
 \|(n_m^{m-1}-1)_+\|_{L^2(Q_T)}\leq C\|\nabla n_m^{m-1}\|_{L^2(Q_T)}\leq C_T.   
\end{equation*}
As a direct consequence, it holds by $\min\{n_m^{m-1},1\}\leq n_m$ for $m\geq2$ that  
\begin{equation*}
\begin{aligned}
\|P_m\|_{L^2(Q_T)}
\leq&
\frac{m}{m-1}\big(\|(n_m^{m-1}-1)_+\|_{L^2(Q_T)}+\|\min\{n_m^{m-1},1\}\|_{L^2(Q_T)}\big)
\\
\leq&
\frac{m}{m-1}\big(\|(n_m^{m-1}-1)_+\|_{L^2(Q_T)}+\|n_m\|_{L^2(Q_T)}\big)\\
\leq &C_T.
\end{aligned}
\end{equation*}
Similarly, we infer from \eqref{high1} that
\begin{equation*}
\|n_m^m\|_{L^2(Q_T)}\leq  \|(n_m^m-1)_+\|_{L^2(Q_T)}+\|n_m\|_{L^2(Q_T)}\leq C\|\nabla n_m^{m}\|_{L^2(Q_T)}+\|n_m\|_{L^2(Q_T)}\leq  C_T.
\end{equation*}
Hence, \eqref{high4} for $d=2$ is justified.

We use Taylor's expansion of $n_m^{2m}$ around $1$ and get 
\begin{equation*}
n_m^{2m}\geq \frac{2m(2m-1)}{2}(n_m-1)_+^2.
\end{equation*}
Then, it follows from \eqref{high4} that
\begin{equation}\label{nm-1}
\|(n_m-1)_+\|_{L^2(Q_T)}^2\leq \frac{2}{2m(2m-1)}\|n_m^{m}\|_{L^2(Q_T)}^2\leq \frac{C_T}{m^2}.
\end{equation}
Hence, \eqref{high3} holds.
\end{proof}

With the aid of Lemmas \ref{lemma41} and \ref{lemma43}, we obtain the uniform estimate for the time derivative $\partial_t n_m$ as follows.

\begin{lemma}\label{lemma44}Under the assumptions of Proposition \ref{propuee}, for any  $m\geq\max\{d+1,5\}$, the following estimate holds:
\begin{equation}\label{uwe}
\|\partial_t n_m\|_{L^2(0,T;\dot{H}^{-1}(\mathbb{R}^d))}\leq C_T.
\end{equation}
\end{lemma}

\begin{proof}
By means of the equation~$\eqref{eqm}_1$,  one has
\begin{equation}\label{120000}
\begin{aligned}
\|\partial_t n_m&\|_{L^2(0,T;\dot{H}^{-1}(\mathbb{R}^d))}\\
\leq&\|\Delta n_m^m\|_{L^2(0,T;\dot{H}^{-1}(\mathbb{R}^d))}+\|\nabla\cdot(n_m\chi(c_m)\nabla c_m)\|_{L^2(0,T;\dot{H}^{-1}(\mathbb{R}^d))}+ 
\|\nabla\cdot(u_m n_m)\|_{L^2(0,T;\dot{H}^{-1}(\mathbb{R}^d))}
\\
\leq &C\|\nabla n_m^m\|_{L^2(Q_T)}+C\|n_m\chi(c_m)\nabla c_m\|_{L^2(Q_T)}+C\|u_mn_m\|_{L^2(Q_T)},
\end{aligned}
\end{equation}
where  we used $u_m\cdot\nabla n_m=\nabla\cdot(u_m n_m)$ derived from $\nabla\cdot u_m=0$.

In accordance with \eqref{basic1}, \eqref{basic3}, \eqref{basic5} and the
Gagliardo-Nirenberg-Sobolev inequality (Lemma~\ref{t9}), it holds that, for
$d\geq3$,
\begin{equation*}
\begin{aligned}
&\|n_m\chi(c_m)\nabla c_m\|_{L^2(Q_T)}
+\|u_mn_m\|_{L^2(Q_T)}\\
&\quad\leq
C\|n_m\|_{L^{\infty}(0,T;L^{d}(\mathbb{R}^d))}
\Big(
\|\nabla c_m\|_{L^2(0,T;L^{\frac{2d}{d-2}}(\mathbb{R}^d))}
+\|u_m\|_{L^2(0,T;L^{\frac{2d}{d-2}}(\mathbb{R}^d))}
\Big)\\
&\quad\leq
C\|n_m\|_{L^{\infty}(0,T;L^{d}(\mathbb{R}^d))}
\Big(
\|\nabla^2 c_m\|_{L^2(0,T;L^2(\mathbb{R}^d))}
+\|\nabla u_m\|_{L^2(0,T;L^2(\mathbb{R}^d))}
\Big)
\leq C_T.
\end{aligned}
\end{equation*}
Here we used $d\leq m-1$, which follows from $m\geq d+1$, and hence
\[
\|n_m\|_{L^\infty(0,T;L^d(\mathbb{R}^d))}\leq C_T
\]
by interpolation between $L^1(\mathbb{R}^d)$ and
$L^{m-1}(\mathbb{R}^d)$.

For $d=2$, by \eqref{basic1}, \eqref{high1}, \eqref{basic3}, \eqref{basic4},
\eqref{basic5} and the two-dimensional Gagliardo-Nirenberg inequality, we have
\begin{equation*}
\begin{aligned}
\|n_m\|_{L^\infty(0,T;L^4(\mathbb{R}^2))}
\leq C_T.
\end{aligned}
\end{equation*}
Indeed, since $m+1\geq4$ for $m\geq3$, this follows from the interpolation
between $L^1(\mathbb{R}^2)$ and $L^{m+1}(\mathbb{R}^2)$, together with
\eqref{basic1} and \eqref{high1}. Therefore,
\begin{equation*}
\begin{aligned}
&\|n_m\chi(c_m)\nabla c_m\|_{L^2(Q_T)}
+\|u_mn_m\|_{L^2(Q_T)}\\
&\quad\leq
C\|n_m\|_{L^{\infty}(0,T;L^{4}(\mathbb{R}^2))}
\Big(
\|\nabla c_m\|_{L^2(0,T;L^{4}(\mathbb{R}^2))}
+\|u_m\|_{L^2(0,T;L^{4}(\mathbb{R}^2))}
\Big)\\
&\quad\leq
C\|n_m\|_{L^{\infty}(0,T;L^{4}(\mathbb{R}^2))}
T^{\frac{1}{4}}
\Big(
\|\nabla c_m\|_{L^{\infty}(0,T;L^2(\mathbb{R}^2))}^{\frac{1}{2}}
\|\nabla^2 c_m\|_{L^2(Q_T)}^{\frac{1}{2}}\\
&\quad\quad
+\|u_m\|_{L^{\infty}(0,T;L^{2}(\mathbb{R}^2))}^{\frac{1}{2}}
\|\nabla u_m\|_{L^2(Q_T)}^{\frac{1}{2}}
\Big)
\leq C_T.
\end{aligned}
\end{equation*}
Putting the above estimates and \eqref{high1} into \eqref{120000} leads to
\eqref{uwe}.
\end{proof}

\subsection{Proof of the Hele-Shaw limit} 
We are going to prove Theorem \ref{thm2}. Based on the uniform regularity estimates obtained in Proposition \ref{propuee}, we can obtain the corresponding convergences with respect to $m$ in \eqref{uml}, in which the limits satisfy the Hele-Shaw type system \eqref{eq1-4}-\eqref{ilr}.

\noindent
\begin{proof}[\textbf{\underline{Justification of \eqref{uml}}}] We first explain $\eqref{uml}_1$-$\eqref{uml}_3$ of $c_m$, $u_m$ and $\Pi_m$. Let $(p_1,q_1):=(\frac{2(d-1)}{d},\frac{d-1}{d-2})$ for $d\geq3$ and $(p_1,q_1)=(\frac{3}{2},3)$ for $d=2$. The uniform estimates obtained in Proposition \ref{propuee} indicate that there exist $c_\infty$, $u_\infty$ and $\Pi_\infty$ such that as $m\rightarrow \infty$, up to subsequences, it holds that
\begin{align}
c_m \rightharpoonup c_{\infty}\quad
&\text{weakly}\quad \text{in}\quad L^2(0,T;H^2(\mathbb{R}^d)),\label{441}\\
u_m \rightharpoonup u_{\infty}\quad
&\text{weakly}\quad \text{in}\quad L^2(0,T;H^1(\mathbb{R}^d)),\label{uweak}\\
\Pi_m \rightharpoonup \Pi_{\infty}\quad
&\text{weakly}\quad \text{in}\quad L^{p_1}(0,T;L^{q_1}(\mathbb{R}^d)),\label{piweak}
\end{align}
which proves $\eqref{uml}_3$. In light of the time derivative estimates in Proposition \ref{propuee} and the Aubin-Lions-Simon lemma (Lemma \ref{lemmaAubin}), as $m\rightarrow \infty$, one also has
\begin{align}
c_m \to c_{\infty}\quad
&\text{strongly}\quad \text{in}\quad L^2(0,T;W^{1,p}_{\mathrm{loc}}(\mathbb{R}^d)),\label{cH1}\\
u_m \to u_{\infty}\quad
&\text{strongly}\quad \text{in}\quad L^2(0,T;L^2_{\mathrm{loc}}(\mathbb{R}^d)),\label{uL2}
\end{align}
for any $p\in(1,\frac{2d}{d-2})$, and thus $\eqref{uml}_{1-2}$ is verified. In addition, since $c_m$ is uniformly bounded in $L^{\infty}(0,T;H^1(\mathbb{R}^d))$, employing the Aubin-Lions-Simon lemma (Lemma \ref{lemmaAubin}) again leads to
\begin{align}
c_m\rightarrow c_\infty\quad &\text{strongly~~in ~~} \mathcal{C}([0,T];L_{\rm{loc}}^{2}(\mathbb{R}^d)),\label{cH0}
\end{align}
which, together with the fact that $0\leq c_m, c_\infty\leq c_B$ and $L^q$ interpolation, yields
\begin{align}
c_m \rightarrow c_{\infty}\quad
&\text{strongly}\quad \text{in}\quad L^{\infty}(0,T;L^q_{\mathrm{loc}}(\mathbb{R}^d)),\label{447}\\
\chi(c_m)\rightarrow \chi(c_{\infty})\quad
&\text{strongly}\quad \text{in}\quad L^{\infty}(0,T;L^q_{\mathrm{loc}}(\mathbb{R}^d)),\label{cH2}
\end{align}
for any $q\in[1,\infty)$.

Next, we are in a position to prove the convergence property $\eqref{uml}_{4-5}$. 
Due to the uniform bounds in Proposition \ref{propuee}, after the extraction of subsequences, there exist two limits $P_\infty$ and $Q_\infty$ in $L^2(0,T;H^1(\mathbb{R}^d))$ such that $P_m$ and $n_m^{m}$, respectively, converge weakly to $P_\infty$ and $Q_\infty$ in $L^2(0,T;H^1(\mathbb{R}^d))$ as $m\to\infty$. By Young's inequality, we discover
\begin{equation*}
    P_m\leq \Big(\frac{m}{m-1}\Big)^{m}\frac{1}{m}+\frac{m-1}{m}n_m^{m},
\end{equation*}
which implies 
\begin{equation*}
    P_\infty\leq Q_\infty.
\end{equation*}
Conversely, for any $\eta>0$, we have
\begin{equation*}
n_m^{m}=\chi_{\{n_m<1+\eta\}}n_m^{m}+\chi_{\{n_m\geq 1+\eta\}}n_m^{m}
\leq (1+\eta)\frac{m-1}{m}P_m+\frac{n_m^{2m}}{(1+\eta)^{m}}.
\end{equation*}
Using the uniform $L^1(Q_T)$-bound of $n_m^{2m}$ \eqref{high4} and passing to the limit as $m\to\infty$, it holds in the sense of distributions that
\begin{equation*}
    Q_\infty\leq (1+\eta)P_\infty.
\end{equation*}
Due to the arbitrariness of $\eta>0$, it follows that 
\begin{equation*}
    Q_\infty\leq P_\infty.
\end{equation*}
Hence, we have $Q_\infty=P_\infty$.  As $m\rightarrow \infty$, it holds true that
\begin{align}
P_m\rightharpoonup  P_{\infty}\quad\text{and}\quad n_m^m \rightharpoonup P_{\infty} \quad& \text{weakly~~in ~~} L^2(0,T;H^1(\mathbb{R}^d)).\label{4.46}
\end{align}

\vspace{1mm}

Then, to justify $\eqref{uml}_{6-7}$, the uniform estimates in Proposition \ref{propuee} guarantee that for any $q\in (1,\infty)$, $n_m$ is uniformly bounded in $L^{\infty}(0,T;L^q(\mathbb{R}^d))$ with respect to $m$ when $m$ is suitably large. As a direct consequence, as $m\rightarrow \infty$, after extraction, there exists a limit $n_\infty$ such that
\begin{align}
n_m\rightharpoonup  n_{\infty} \quad& \text{weakly$^*$~~in ~~} L^{\infty}(0,T;L^q(\mathbb{R}^d)).\label{4.48}
\end{align}
Together with \eqref{uwe} and the Aubin-Lions-Simon lemma (Lemma \ref{lemmaAubin}), this yields
\begin{align}
n_m\rightarrow  n_{\infty} \quad& \text{strongly~~in ~~} L^{2}(0,T;\dot{H}^{-1}_{\mathrm{loc}}(\mathbb{R}^d)).\label{4.49}
\end{align}

Therefore, by \eqref{441}-\eqref{4.49} we justify all the properties in \eqref{uml}.
\end{proof}

\vspace{2mm}

 Based on the convergence properties in \eqref{uml}, we establish the Hele-Shaw system \eqref{eq1-4}-\eqref{ilr}. In previous works~\cite{HLP2023,HLP2022,CKY2018}, the complementarity relation \eqref{cr2} is a major challenge. In this paper, we observe that one can directly derive this relation from the Hele-Shaw structure by choosing the special test functions.

\begin{proof}[\textbf{\underline{Justification of \eqref{eq1-4}-\eqref{ilr} and \eqref{cr2}}}] 
By Definition \ref{ws} and the convergence results \eqref{uml}, one can obtain the Hele-Shaw system \eqref{eq1-4}
in the sense of distributions.

We now prove the Hele-Shaw graph relations \eqref{ilr} expressed by 
\begin{equation}\label{pninfty}
0\leq n_\infty\leq 1,\quad (1-n_\infty)P_\infty=0,\quad (1-n_\infty)\nabla P_\infty=0,\quad\text{a.e. in }Q_T.
\end{equation}
The first estimate of \eqref{pninfty} can be directly derived from the property $\|(n_m-1)_+\|_{L^2(Q_T)}\leq Cm^{-1}$ in Proposition \ref{propuee}.  
As $n_m^m=\frac{m-1}{m}n_m P_m$, after extraction, it follows from the weak-strong convergence properties \eqref{4.46} and \eqref{4.49} that 
\begin{equation*}
n_m^m=\frac{m-1}{m}n_m P_m\rightharpoonup n_\infty P_\infty,\quad \text{in }~~\mathcal{D}'(Q_T)\ \text{as }m\to\infty,
\end{equation*}
which, combined with \eqref{4.46}, yields the second estimate of \eqref{pninfty}. In addition, as observed in \cite{HLP2023}, for any $\alpha>0$ one has
\[
n_\infty P_\infty^\alpha=P_\infty^\alpha
\quad\text{a.e. in }Q_T,
\]
as a consequence of $(1-n_\infty)P_\infty=0$. Hence,  it holds for any $\alpha>0$ that
\begin{equation}\label{upinfty}
{n_\infty}\nabla P_\infty^{1+\alpha}=(1+\alpha)n_\infty P_\infty^\alpha  \nabla P_\infty=(1+\alpha) P_\infty^\alpha \nabla P_\infty=\nabla P_\infty^{1+\alpha}.
\end{equation}
If $0<\alpha\leq \frac{1}{3}$, we have $\nabla P_\infty^{1+\alpha}\in L_{{\rm{loc}}}^{\frac{3}{2}}(Q_T)$ due to $P_\infty\in L^2(0,T;H^1(\mathbb{R}^d))$, and further conclude that
\[\nabla P_\infty^{\alpha+1}\rightharpoonup\nabla P_\infty,\quad \text{ in }L_{{\rm{loc}}}^{\frac{3}{2}}(Q_T)\quad \text{as }\alpha\to0^+.\]
Therefore, the third estimate of \eqref{pninfty}  holds after taking the limit $\alpha\to 0^+$ for \eqref{upinfty}.   

Using $\eqref{eq1-4}_1$ and direct computations, we obtain
\begin{equation}\label{hsh-1}
    \begin{aligned}
\|\partial_tn_\infty&\|_{L^2(0,T;\dot{H}^{-1}(\mathbb{R}^d))}\\
\leq& \|\nabla P_\infty\|_{L^2(Q_T)}+\||u_\infty|n_\infty\|_{L^2(Q_T)}+\|n_\infty\chi(c_\infty)|\nabla c_\infty|\|_{L^2(Q_T)}\leq C_T.       
    \end{aligned}
\end{equation}
For any $\varphi\in \mathcal{C}_0^1(Q_T)$ with $\varphi\geq0$,  for any $h>0$ and all $d\geq2$, we have
\[\frac{n_\infty(t+h)-n_\infty(t)}{h}\varphi(t)P_\infty(t)=\frac{n_\infty(t+h)-1}{h}\varphi(t)P_\infty(t)\leq 0,\]
and
\[\frac{n_\infty(t)-n_\infty(t-h)}{h}\varphi(t)P_\infty(t)=\frac{1-n_\infty(t-h)}{h}\varphi(t)P_\infty(t)\geq 0.\]
Since $\partial_t n_\infty\in L^2(0,T;\dot{H}^{-1}(\mathbb{R}^d))$ on account of \eqref{hsh-1} and $\varphi P_\infty\in L^2(0,T;\dot{H}^1(\mathbb{R}^d))$ forms a duality pairing, we take the limit as $h\to0^+$ similarly in \cite[page 21]{noemi2023} and then obtain
\begin{equation}\label{hsh-1p}
\iint_{Q_T}\partial_t n_\infty \varphi P_\infty\, dxdt=0.
\end{equation}
Thus, by a density argument, one can take $\varphi P_\infty$ as the test function for the Hele-Shaw problem $\eqref{eq1-4}_1$ to obtain
\begin{equation*}
\begin{aligned}
    \iint_{Q_T}\partial_t n_\infty \varphi P_\infty\, dxdt-\iint_{Q_T}&n_\infty u_\infty\cdot\nabla(\varphi P_\infty) \,dxdt\\=&-\iint_{Q_T}\nabla P_\infty\cdot\nabla (\varphi P_\infty)+\iint_{Q_T}n_\infty\chi(c_\infty)\nabla c_\infty\cdot\nabla (\varphi P_\infty) \,dxdt.
\end{aligned}
\end{equation*}
Using the Hele-Shaw graph \eqref{pninfty} and the divergence-free condition $\nabla\cdot u_\infty=0$, we also have
\[
\iint_{Q_T}n_\infty u_\infty\cdot\nabla(\varphi P_\infty) \,dxdt=\iint_{Q_T} u_\infty\cdot\nabla(\varphi P_\infty) \,dxdt=-\iint_{Q_T}\nabla\cdot u_\infty \varphi P_\infty\, dxdt=0,
\]
and
\[\iint_{Q_T}n_\infty\chi(c_\infty)\nabla c_\infty\cdot\nabla (\varphi P_\infty) \,dxdt=\iint_{Q_T}\chi(c_\infty)\nabla c_\infty\cdot\nabla (\varphi P_\infty) \,dxdt.\]
We obtain
\begin{equation}\label{scr}
-\iint_{Q_T}\nabla P_\infty\cdot\nabla (\varphi P_\infty)\,dxdt+\iint_{Q_T}\chi(c_\infty)\nabla c_\infty\cdot\nabla (\varphi P_\infty) \,dxdt=0.
\end{equation}
Similarly, \eqref{hsh-1p} and \eqref{scr} hold for any $\varphi\in \mathcal{C}_0^1(Q_T)$ with $\varphi\leq0$. Consequently, \eqref{scr} follows for any $\varphi\in \mathcal{C}_0^\infty(Q_T)$. This verifies the complementarity relation \eqref{cr2} in the sense of distributions. 
\end{proof}

\begin{appendices}

\section{Another proof of the complementarity relation}\label{appendix1}
  As a comparison and to understand the nonlinear diffusion more, we show the complementarity relation \eqref{cr2} by passing to the limit in the equation $\eqref{eqm}_1$ tested by $ n_m^{m}$ as the diffusion exponent $m\to\infty$. The proof of the complementarity relation \eqref{cr2} is equivalent to proving the strong convergence of $\{\nabla n_m^m\}_{m>1}$ in $L^2(Q_T)$. In this section, we make full use of the special structure of the porous medium type equation as achieved in \cite{noemi2023,LX2021} for tumor (tissue) growth and in \cite{HLP2023,HZ2024} for chemotaxis. To this end, we need the following additional assumptions to establish some further regularity estimates:
  \begin{equation}\label{H4}\tag{${\rm{H}}_4$}
\begin{aligned}
&\|n_{m,0}\|_{L^{m+3}(\mathbb{R}^d)}^{m+3}\leq C,\quad \| |x|^2 n_{m,0}\|_{L^1(\mathbb{R}^d)}\leq C,\quad \||x|c_{m,0}\|_{L^2(\mathbb{R}^d)}\leq C,
\end{aligned}
\end{equation}
for some constant $C>0$ independent of $m$.

\begin{lemma}\label{ka} 
Let $(n_m,c_m,u_m)$ be a weak solution for the Cauchy problem \eqref{eqm} obtained in Theorem \ref{thm1} with $m\geq \max\{2d+1,9\}$, and let $\Pi_m$ be the pressure given by \eqref{pressure}, and set $P_m:=\frac{m}{m-1} n_m^{m-1}$. Then, under the assumptions \eqref{ca}, \eqref{a1}, \eqref{H3} and \eqref{H4}, for $t\in[0,T]$, we have
\begin{align}
&\|n_m(t)\|_{L^{m+3}(\mathbb{R}^d)}+\|\nabla n_m^{m+1}\|_{L^2(Q_T)}\leq C_T,\label{gradnm+1}\\
&\|n_m^{m+1}\|_{L^2(0,T;L^2(\mathbb{R}^d))}\leq C_T, \label{nm+1}\\
&\|n_m(t)|x|^2\|_{L^1(\mathbb{R}^d)}\leq C_T,\label{cwe1}\\
&\|c_m(t)|x|\|_{L^2(\mathbb{R}^d)}+\||\nabla c_m||x|\|_{L^2(Q_T)}\leq C_T.\label{cwe}
\end{align}
\end{lemma}
\begin{proof}
Following the same line of \eqref{high1}, one can show \eqref{gradnm+1}.  Under the condition $m\geq \max\{2d+1,9\}$, by choosing $n_m^{m+2}$ as a test function, the two estimates of \eqref{nm+1} are obtained by similar arguments to \eqref{high4}. The details are omitted.

 To show \eqref{cwe1}, we multiply $\eqref{eqm}_1$  by $|x|^2$ and integrate on $\mathbb{R}^d$, and then attain 
 \begin{equation*}
 \begin{aligned}
     \frac{1}{2}\frac{d}{dt}\int_{\mathbb{R}^d}&n_m|x|^2\,dx\\
     =&\int_{\mathbb{R}^d}n_m u_m\cdot x\,dx+d\int_{\mathbb{R}^d}n_m^m\,dx+\int_{\mathbb{R}^d}n_m\chi(c_m)\nabla c_m\cdot x\,dx\\
     \leq& \int_{\mathbb{R}^d}n_m|x|^2\,dx+C\int_{\mathbb{R}^d}n_m(|u_m|^2+|\nabla c_m|^2)dx+d\int_{\mathbb{R}^d}n_m^m\,dx\\
     \leq &\int_{\mathbb{R}^d}n_m|x|^2\,dx+d\int_{\mathbb{R}^d}n_m^m\,dx\\
     &+\begin{cases}
     C\|n_m\|_{L^{\frac{d}{2}}(\mathbb{R}^d)}(\|\nabla u_m\|_{L^2(\mathbb{R}^d)}^2+\|\nabla^2 c_m\|_{L^2(\mathbb{R}^d)}^2), &d\geq3,\\
      C\|n_m\|_{L^2(\mathbb{R}^2)}(\| u_m\|_{L^2(\mathbb{R}^2)}\|\nabla u_m\|_{L^2(\mathbb{R}^2)}+\|\nabla c_m\|_{L^2(\mathbb{R}^2)}\|\nabla^2 c_m\|_{L^2(\mathbb{R}^2)}),&d=2,
     \end{cases}
\end{aligned}
 \end{equation*}
where the Gagliardo-Nirenberg-Sobolev inequality (Lemma~\ref{t9}) has been used. Together with \eqref{basic1}, \eqref{basic4}, \eqref{basic5} and Gr\"onwall's inequality, this leads to  \eqref{cwe1}.

Finally, multiplying $\eqref{eqm}_2$ by $c_m|x|^2$ and integrating on $\mathbb{R}^d$, it holds by integrating by parts that 
\begin{equation*}
\begin{aligned}
\frac{1}{2}\frac{d}{dt}\int_{\mathbb{R}^d}&c_m^2|x|^2\,dx+\int_{\mathbb{R}^d}|\nabla c_m|^2|x|^2\,dx\\
=&d\int_{\mathbb{R}^d}c_m^2\,dx-\int_{\mathbb{R}^d}n_mf(c_m)c_m|x|^2\,dx+\int_{\mathbb{R}^d}c_m^2 u_m\cdot x\,dx\\
\leq &C+C\int_{\mathbb{R}^d}n_m |x|^2 \,dx+\frac{c_B^2}{8}\int_{\mathbb{R}^d}c_m^2|x|^2\,dx+\int_{\mathbb{R}^d}|u_m|^2\,dx\\
\leq& C_T+C\int_{\mathbb{R}^d}c_m^2|x|^2\,dx,\end{aligned}
\end{equation*}
where we used  \eqref{basic3}, \eqref{basic5} and \eqref{cwe1}. Consequently, \eqref{cwe} holds by means of the initial assumption  \eqref{H4} and Gr\"onwall's inequality, and the proof of Lemma \ref{ka} is finished.
\end{proof}

\begin{lemma}
Under the assumptions of Lemma \ref{ka}, it holds for any $p\in[2,\frac{2d}{d-2})$ and $q\in [1,\infty)$ that
\begin{alignat}{4}
n_m^{m+1} &\rightharpoonup P_\infty
&\;\;\;& \text{weakly}
&\;\;\; & \text{in }
&\;\;\; & L^2(0,T;H^1(\mathbb{R}^d)), \label{rd}\\
c_m &\to c_\infty
&\;\;\;& \text{strongly}
&\;\;\; & \text{in }
&\;\;\; & L^2(0,T;W^{1,p}(\mathbb{R}^d))
      \cap L^{\infty}(0,T;L^q(\mathbb{R}^d)). \label{cm2}
\end{alignat}
\end{lemma}

\begin{proof}
\eqref{rd}  can be directly proved by \eqref{gradnm+1}-\eqref{nm+1} and a similar argument as in \eqref{4.46}-\eqref{4.48}. Due to \eqref{cwe} and Fatou's property, it holds that
\begin{align*}
\sup\limits_{0\leq t\leq T}\|c_\infty(t)|x|\|_{L^2(\mathbb{R}^d)}+\||\nabla c_\infty||x|\|_{L^2(Q_T)}\leq C_T,
\end{align*}
from which and \eqref{cwe} we infer
\begin{align*}
\iint_{Q_T} (|c_m-c_\infty|^2+|\nabla(c_m-c_\infty)|^2) \,dxdt&\leq \int_0^T \hskip-4pt\int_{B_R} (|c_m-c_\infty|^2+|\nabla(c_m-c_\infty)|^2) \,dxdt+\frac{C_T}{R^2},
\end{align*}
for any $R>0$. Thus, according to \eqref{cH1}, we justify the strong convergence of $c_m$ in $L^2(0,T;H^1(\mathbb{R}^d))$ by first taking the limit as $m\rightarrow \infty$ and then letting $R\rightarrow \infty$. Then, using the Gagliardo-Nirenberg-Sobolev inequality (Lemma~\ref{t9}) and Proposition \ref{propuee}, we further infer for any $p\in (2,\frac{2d}{d-2})$ that
\begin{align*}
\|\nabla(c_m-c_\infty)\|_{L^2(0,T;L^p(\mathbb{R}^d))}\leq C\|\nabla(c_m-c_\infty)\|_{L^2(Q_T)}^{\theta} \|(\nabla^2 c_m, \nabla^2c_\infty)\|_{L^2(Q_T)}^{1-\theta}\rightarrow 0\quad\text{as}\quad m\rightarrow \infty,
\end{align*}
where $\theta\in(0,1)$ is given by $\frac{1}{p}=\frac{1}{2}\theta+(\frac{1}{2}-\frac{1}{d})(1-\theta)$. Similarly, using \eqref{447} and \eqref{cwe}, one has the strong convergence of $c_m$ in $\mathcal{C}([0,T];L^1(\mathbb{R}^d))$. Together with the upper bound of $c_m$ and $L^q$ interpolation, we eventually arrive at \eqref{cm2}.
\end{proof}

We prove the key convergence property of $\nabla n_m^m$ in $L^2(Q_T)$.

 \begin{prop}\label{sg}
    Under the assumptions of Lemma~\ref{ka} with the conditions \eqref{ca}, after the extraction of a subsequence, it holds true that
    \begin{align}
        \nabla n_m^m\to \nabla P_\infty\   \text{ strongly in }L^2(Q_T)\quad \text{ as }m\to\infty.\label{nablan:c}
    \end{align}
\end{prop}
\begin{proof}

We have the difference equation 
\begin{equation}\label{de}
\begin{aligned}
\partial_t(n_m-n_\infty)&+(u_m\cdot\nabla n_m-u_\infty\cdot\nabla n_\infty)\\
=&\Delta (n_m^m-P_\infty)-\nabla\cdot\big(n_m\chi(c_m)\nabla c_m-n_\infty\chi(c_\infty)\nabla c_\infty\big).
\end{aligned}
\end{equation}
 Let $n_m^m$ be the test function for the equation \eqref{de}. Then it follows that
 \begin{equation}\label{ide}
     \begin{aligned}
\iint_{Q_T}&|\nabla(n_m^m-P_\infty)|^2  dxdt\\
\leq& \iint_{Q_T}\partial_t n_\infty n_m^m  dxdt+\frac{1}{m+1}\int_{\mathbb{R}^d}n_{m,0}^{m+1}\,dx\\
&+\iint_{Q_T}\nabla(n_m^m-P_\infty)\cdot\nabla P_\infty\, dxdt\\
&+\iint_{Q_T}\nabla n_m^m\cdot \big( n_m\chi(c_m)\nabla c_m-n_\infty\chi(c_\infty)\nabla c_\infty\big)\,dxdt\\
&+\iint_{Q_T}(u_m n_m-u_\infty n_\infty)\cdot\nabla n_m^m  dxdt.
     \end{aligned}
 \end{equation}
The first term on the right-hand side of \eqref{ide} vanishes as $m\rightarrow\infty$. Indeed, for any $h>0$ and all $d\geq2$, we have
\[\frac{n_\infty(t+h)-n_\infty(t)}{h}P_\infty(t)=\frac{n_\infty(t+h)-1}{h}P_\infty(t)\leq 0,\]
\[\frac{n_\infty(t)-n_\infty(t-h)}{h}P_\infty(t)=\frac{1-n_\infty(t-h)}{h}P_\infty(t)\geq 0.\]
On account of the facts that $\partial_t n_\infty\in L^2(0,T;\dot{H}^{-1}(\mathbb{R}^d))$ are obtained by \eqref{hsh-1} and $P_\infty\in L^2(0,T;\dot{H}^1(\mathbb{R}^d))$ from \eqref{rd}, one takes the limit as $h\to0^+$ in the sense of duality (see \cite[Page 21]{noemi2023}) and then obtains 
$$\iint_{Q_T}\partial_t n_\infty P_\infty\, dxdt=0.
$$
Consequently, it further holds by $\|n_m^m\|_{L^2(0,T;\dot{H}^1(\mathbb{R}^d))}\leq C_T$ for $d\geq 2$ and the duality relation that 
\begin{equation}\label{k1}\iint_{Q_T}\partial_t n_\infty n_m^m  \,dxdt\to \iint_{Q_T}\partial_t n_\infty P_\infty\, dxdt=0\quad\text{ as }m\to\infty. 
\end{equation}
Concerning the second term,  by means of \eqref{4.46}, we have
\begin{equation}\label{k2}\iint_{Q_T}\nabla(n_m^m-P_\infty)\cdot\nabla P_\infty\, dxdt\to\iint_{Q_T}\nabla(P_\infty-P_\infty)\cdot\nabla P_\infty\, dxdt=0\quad \text{ as }m\to\infty.
\end{equation}
Recalling  $\nabla\cdot u_m=0$, $\nabla\cdot u_\infty=0$, the Hele-Shaw graph relations \eqref{pninfty} and \eqref{4.46}, we obtain
\begin{equation}\label{k3}
\begin{aligned}
&\iint_{Q_T}(u_m n_m-u_\infty n_\infty)\cdot\nabla n_m^m\,dxdt\\
&\quad=\frac{m}{m+1}\iint_{Q_T} u_m\cdot\nabla n_m^{m+1}dxdt-\iint_{Q_T} n_\infty u_\infty\cdot\nabla n_m^mdxdt\\
&\quad=-\iint_{Q_T}n_\infty u_\infty\cdot\nabla n_m^m dxdt\\
&\quad\rightarrow -\iint_{Q_T}n_\infty u_\infty\cdot\nabla P_\infty dxdt=-\iint_{Q_T} u_\infty\cdot\nabla P_\infty dxdt=0\quad\text{as }m\to\infty.
\end{aligned}
\end{equation}
In view of the strong convergence \eqref{cm2} and the structural conditions \eqref{ca}, it follows that
\begin{equation}\label{sc}
\chi(c_m)\nabla c_m \to \chi(c_\infty)\nabla c_{\infty}\quad \text{in }L^2(Q_T)\quad \text{ as }m\to\infty.   \end{equation}
Furthermore, using \eqref{rd},  \eqref{sc} and \eqref{4.46},  it holds by the weak-strong convergence that 
\begin{equation}\label{k4}
\begin{aligned}\iint_{Q_T}&\nabla n_m^m\cdot\big( n_m\chi(c_m)\nabla c_m-n_\infty\chi(c_\infty)\nabla c_\infty\big)\, dxdt\\
&=\iint_{Q_T}\Big(\frac{m}{m+1}\chi(c_m)\nabla c_m\cdot\nabla n_m^{m+1}-n_\infty\chi(c_\infty)\nabla c_\infty\cdot\nabla n_m^m \big) \,dxdt\\
&\to \iint_{Q_T}\chi(c_\infty)\nabla c_\infty\cdot\nabla P_\infty-n_\infty\chi(c_\infty)\nabla c_\infty\cdot\nabla P_\infty\, dxdt\\
&=0\quad \text{ as  }m\to\infty,
\end{aligned}\end{equation} 
where we used $n_\infty\nabla P_\infty=\nabla P_\infty$. Moreover, by the assumption \eqref{H3}, we have
\[
\frac{1}{m+1}\int_{\mathbb{R}^d}n_{m,0}^{m+1}\,dx\to0
\quad\text{as }m\to\infty.
\]
Substituting \eqref{k1}, \eqref{k2}, \eqref{k3}, \eqref{k4} and the above
initial estimate into \eqref{ide}, we obtain \eqref{nablan:c}.
\end{proof}

\vspace{1mm}

\noindent\underline{\textbf{\emph{Proof of the complementarity relation.}}}\ \     Let $n_m^m\varphi$ be a test function with any $\varphi\in \mathcal{C}_0^\infty(Q_T)$ for $\eqref{eqm}_1$, then we have  
\begin{equation*}
\begin{aligned}-\iint_{Q_T}&\frac{n_m^{m+1}}{m+1}\partial_t\varphi  \,dxdt+\iint_{Q_T}\big(|\nabla n_m^m|^2\varphi+n_m^m\nabla n_m^m\cdot\nabla \varphi\big) \,dxdt\\
&=\frac{1}{m+1}\iint_{Q_T}n_m^{m+1}u_m\cdot\nabla\varphi  \,dxdt\\
&\quad +\iint_{Q_T}\Big(\frac{m}{m+1}\chi(c_m)\nabla c_m\cdot\nabla n_m^{m+1}\varphi+n_m^{m+1}\chi(c_m)\nabla c_m\cdot\nabla \varphi\Big) \,dxdt.
\end{aligned}
\end{equation*}
By means of the convergence properties \eqref{rd},\eqref{nablan:c}, \eqref{sc} and \eqref{4.46} as well as the regularity estimates \eqref{basic5} and \eqref{nm+1}, after passing to the limit as $m\rightarrow\infty$,  one deduces that
\begin{equation*}
\begin{aligned}&\iint_{Q_T}\big(|\nabla P_\infty|^2\varphi+P_\infty\nabla P_\infty\cdot\nabla \varphi\big) \,dxdt-\iint_{Q_T}\big(\chi(c_\infty)\nabla c_\infty\cdot\nabla P_\infty\varphi+P_\infty\chi(c_\infty)\nabla c_\infty\cdot\nabla \varphi\big) \,dxdt=0.
\end{aligned}
\end{equation*}
Hence, the complementarity relation \eqref{cr2} holds in the sense of distributions. \hfill $\Box$

\section{Proof of Proposition \ref{propapp}}\label{App:proofapp}

For any $\eta\in(0,1)$, we consider the following regularized problem
\begin{equation}\label{app1}
\left\{
\begin{aligned}
&\partial_t n_\eta+(J_{\eta}\ast u_\eta)\cdot \nabla n_\eta=m\nabla\cdot \Big(\big( n_\eta^{m-1}\ast J_{\eta}\big)\nabla n_\eta\Big)+\varepsilon  \Delta n_\eta -\nabla \cdot \big(n_\eta\chi(c_\eta) \nabla( J_{\varepsilon}\ast c_\eta)\big),\\
&\partial_t c_\eta+ (J_{\eta}\ast u_\eta) \cdot \nabla c_\eta=\Delta c_\eta-n_\eta f(c_\eta),\\
&\partial_t u_\eta +(J_{\eta}\ast u_\eta) \cdot \nabla  u_\eta+\nabla \Pi_\eta
=\Delta u_\eta-n_\eta\nabla (J_{\eta}\ast \phi),\\
&\nabla \cdot u_{\eta}=0,\\
&(n_{\eta},c_{\eta},u_{\eta})(x,0)=(n_{0,\varepsilon},c_{0,\varepsilon}, u_{0,\varepsilon})(x).
\end{aligned}
\right.
\end{equation}
where the initial data $(n_{0,\varepsilon},c_{0,\varepsilon}, u_{0,\varepsilon})$ with $0<\varepsilon<1$ is given by \eqref{appd}, and $J_{\eta}$ and $J_{\varepsilon}$ denote standard mollifiers.

For fixed $0<\varepsilon,\eta <1$, there exists a time $T_\eta$ such that the regularized problem \eqref{app1} has a unique strong solution $(n_\eta,c_\eta,u_\eta)\in C([0,T_{\eta});H^{s_*}(\mathbb{R}^d))$ with $s_*\geq [\frac{d}{2}]+1$. Since the proof of local existence follows a quite standard way, we omit the details for brevity; cf.~\cite{ckl2014,heshouwu}. Due to the property of the mollifier on every nonlinear term in \eqref{app1}, one can prove the uniform-in-time a priori estimates and extend the local solution globally in time. Thus, we obtain a global approximate sequence $\{(n_{\eta},c_{\eta},u_{\eta})\}_{0<\eta<1}$.


Next, we establish the uniform-in-$\eta$ estimates of the global strong solution $(n_\eta,c_\eta,u_\eta)$ for any $0<\eta<1$. First, it is clear that $0\leq c_\eta \leq c_B$ and $\|c_\eta\|_{L^1(\mathbb{R}^d)}\leq \|c_{0,\varepsilon}\|_{L^1}$. Young's inequality for convolutions yields $\|\nabla (J_\varepsilon\ast c_\eta)\|_{L^2(\mathbb{R}^d)}\leq C_\varepsilon  \|c_{0,\varepsilon}\|_{L^1(\mathbb{R}^d)}$.  Via the standard $L^2$ estimate for the parabolic equations, it holds by the Cauchy-Schwarz inequality that
\begin{equation*}
\begin{aligned}
&\frac{1}{2}\frac{d}{dt}\| (n_{\eta},c_{\eta},u_{\eta})\|_{L^2(\mathbb{R}^d)}^2+\int_{\mathbb{R}^d}\Big( \big(\var+m  n_\eta^{m-1}\ast J_{\eta}  \big) |\nabla n_\eta|^2+|\nabla c_{\eta}|^2+|\nabla u_\eta|^2+n_\eta f(c_\eta) c_\eta\Big)\,dx\\
&=\int_{\mathbb{R}^d}\Big(\chi(c_\eta) n_{\eta} \nabla(J_\varepsilon  \ast c_\eta) \cdot \nabla n_\eta-n_\eta u_\eta\cdot \nabla( J_\eta\ast \phi)\Big)\,dx\\
&\leq \frac{\varepsilon}{4} \| \nabla n_\eta\|_{L^2(\mathbb{R}^d)}^2+C_\varepsilon\Big(\sup_{0\leq s\leq c_B}\chi(s)^2\|c_{0,\varepsilon}\|_{L^1(\mathbb{R}^d)}^2+\|\nabla\phi\|_{L^{\infty}}^2\Big)\|n_{\eta}\|_{L^2(\mathbb{R}^d)}^2+C\|u_\eta\|_{L^2(\mathbb{R}^d)}^2.
\end{aligned}
\end{equation*}
Then Grönwall's inequality ensures that
\begin{equation}\label{etaL2}
\begin{aligned}
\sup_{t\in[0,T]}\| (n_{\eta},c_{\eta},u_{\eta})\|_{L^2(\mathbb{R}^d)}^2+\int_0^T \|(\nabla n_{\eta},\nabla c_{\eta},\nabla u_{\eta})\|_{L^2(\mathbb{R}^d)}^2\,dt\leq C_{\varepsilon,T}.
\end{aligned}
\end{equation}
Moreover, from $\eqref{app1}_2$, one arrives at
\begin{equation}
\begin{aligned}
\sup_{t\in[0,T]}\| \nabla c_{\eta} \|_{L^2(\mathbb{R}^d)}^2+\int_0^T \|\nabla^2 c_{\eta}\|_{L^2(\mathbb{R}^d)}^2\,dt\leq C_{\varepsilon,T}.\label{etacH1}
\end{aligned}
\end{equation}
Proving this inequality for $c$ is totally similar to that of \eqref{nablacL2var}. Next, multiplying $\eqref{app1}_1$ with $q n_\eta^{q-1}$ for $m-1<q<\infty$, integrating the resulting equation over $\mathbb{R}^d$ and using $\nabla\cdot u_\eta=0$, we have
\begin{equation}\nonumber
\begin{aligned}
\frac{d}{dt}&\|n_\eta\|_{L^q(\mathbb{R}^d)}^q+\varepsilon  q(q-1)\int_{\mathbb{R}^d} n_\eta^{q-2}|\nabla n_\eta |^2 \,dx+m q(q-1) \int_{\mathbb{R}^d} \big( n_\eta^{m-1}\ast J_{\eta}\big) n_\eta^{q-2} |\nabla n_\eta |^2 \,dx\\
&=q(q-1)\int_{\mathbb{R}^d} \chi(c_\eta) n_\eta^{q-1} \nabla(J_\varepsilon\ast c_\eta ) \cdot \nabla n_\eta \,dx\\
&\leq \frac{1}{2}\varepsilon  q(q-1)\int_{\mathbb{R}^d} n_\eta^{q-2}|\nabla n_\eta |^2 \,dx+C_\varepsilon  q(q-1) \|n_\eta\|_{L^q(\mathbb{R}^d)}^q.
\end{aligned}
\end{equation}
Here we have used $\|\nabla(J_\varepsilon\ast c_\eta)\|_{L^{\infty}(\mathbb{R}^d)}\leq C_\varepsilon  \|c_\eta\|_{L^{\infty}(\mathbb{R}^d)}\leq C_{\varepsilon} $. Therefore, combining with $\|n_\eta\|_{L^1}=\|n_{0,\varepsilon}\|_{L^1}$, we deduce \begin{align}
\sup\limits_{t\in[0,T]}\|n_\eta\|_{L^q(\mathbb{R}^d)}\leq  C_{\varepsilon,T},\quad 1\leq q<\infty.\label{nLp}
\end{align}

With the aid of \eqref{etaL2}, \eqref{etacH1} and \eqref{nLp}, a limit $(n_\varepsilon,c_\varepsilon,u_\varepsilon)$ exists such that as $\eta\rightarrow 0$, for any time $T>0$, it holds up to subsequences that
\begin{equation}\label{weak:app}
\left\{
\begin{alignedat}{4}
n_\eta \rightharpoonup n_\varepsilon  \quad
&\text{weakly}^\ast \quad &\text{in}\quad &L^{\infty}(0,T;L^{p}(\mathbb{R}^d)),\qquad &1\le p<\infty,\\
n_\eta \rightharpoonup n_\varepsilon  \quad
&\text{weakly}       \quad &\text{in}\quad &L^2(0,T;H^1(\mathbb{R}^d)),\\
c_\eta \rightharpoonup c_\varepsilon  \quad
&\text{weakly}^\ast \quad &\text{in}\quad &L^{\infty}(0,T;L^{\infty}(\mathbb{R}^d)\cap H^1(\mathbb{R}^d)),\\
c_\eta \rightharpoonup c_\varepsilon  \quad
&\text{weakly}       \quad &\text{in}\quad &L^2(0,T;H^2(\mathbb{R}^d)),\\
u_\eta \rightharpoonup u_\varepsilon  \quad
&\text{weakly}^\ast \quad &\text{in}\quad &L^{\infty}(0,T;L^2(\mathbb{R}^d)),\\
u_\eta \rightharpoonup u_\varepsilon  \quad
&\text{weakly}       \quad &\text{in}\quad &L^2(0,T;H^1(\mathbb{R}^d)).
\end{alignedat}
\right.
\end{equation}
To justify the strong convergence, one needs to estimate the time derivatives. Arguing similarly as for proving \eqref{cutvar}, for $s_0>\frac{d}{2}$, we can obtain
\begin{equation*}
\begin{aligned}
&\|\partial_t c_\eta\|_{L^2(0,T;H^{-1}(\mathbb{R}^d))}+\|\partial_t u_\eta\|_{L^{2}(0,T;H^{-s_0-1}(\mathbb{R}^d))}\leq C_{\varepsilon,T}.
\end{aligned}
\end{equation*}
For any $\varphi\in L^{2}(0,T;H^{s_0+2}(\mathbb{R}^d))$, one has
\begin{equation*}
\begin{aligned}
\left|\int_0^T\hskip-4pt\int_{\mathbb{R}^d} \partial_t n_\eta  \varphi  \,dxdt\right|
\leq& \|n_\eta^{m-1}\|_{L^{\infty}(0,T;L^{2}(\mathbb{R}^d))}
\|\nabla n_\eta\|_{L^2(0,T;L^2(\mathbb{R}^d))}
\|\Delta \varphi\|_{L^{2}(0,T;L^{\infty}(\mathbb{R}^d))}\\
&+\varepsilon \|n_\eta\|_{L^{\infty}(0,T;L^2(\mathbb{R}^d))}
\|\Delta \varphi\|_{L^{1}(0,T;L^{2}(\mathbb{R}^d))}\\
&+\|u_\eta\|_{L^{\infty}(0,T;L^2(\mathbb{R}^d))}
\|n_\eta\|_{L^{\infty}(0,T;L^2(\mathbb{R}^d))}
\|\nabla\varphi\|_{L^1(0,T;L^{\infty}(\mathbb{R}^d))}\\
&+\sup_{0\leq s\leq c_B}|\chi(s)|
\|n_\eta\|_{L^{\infty}(0,T;L^2(\mathbb{R}^d))}
\|\nabla c_\eta\|_{L^{\infty}(0,T;L^2(\mathbb{R}^d))}
\|\nabla\varphi\|_{L^1(0,T;L^{\infty}(\mathbb{R}^d))}\\
\leq& C_{\varepsilon,T}\|\varphi\|_{L^{2}(0,T;H^{s_0+2}(\mathbb{R}^d))}.
\end{aligned}
\end{equation*}
which implies 
\[ \|\partial_t n_\eta\|_{L^2(0,T;H^{-2-s_0}(\mathbb{R}^d))}\leq C_{\varepsilon,T}. \]
Therefore, up to a subsequence,  as $\eta\rightarrow 0$, we use  the Aubin-Lions-Simon lemma (Lemma~\ref{lemmaAubin}) and obtain 
\begin{equation}\label{strongap}
\left\{
\begin{alignedat}{3}
n_\eta \to n_\varepsilon  \quad
&\text{strongly}\quad &\text{in}\quad &L^2([0,T];L_{\mathrm{loc}}^{2}(\mathbb{R}^d)),\\
c_\eta \to c_\varepsilon  \quad
&\text{strongly}\quad &\text{in}\quad &L^2([0,T];H_{\mathrm{loc}}^{1}(\mathbb{R}^d)),\\
u_\eta \to u_\varepsilon  \quad
&\text{strongly}\quad &\text{in}\quad &L^2([0,T];L_{\mathrm{loc}}^{2}(\mathbb{R}^d)).
\end{alignedat}
\right.
\end{equation}

Note that the strong convergence property $\eqref{strongap}_3$ implies $J_{\eta}\ast u_\eta-u_\varepsilon=J_{\eta}\ast u_\eta-J_{\eta}\ast u_\varepsilon+J_{\eta}\ast u_\varepsilon-u_\varepsilon\rightarrow 0 $ in $L^2(0,T;L^2_{{\rm{loc}}}(\mathbb{R}^d))$ as $\eta\to0$. Consequently, together with \eqref{weak:app}, as $\eta\to0$, it holds by the weak-strong convergence  that 
\begin{equation*}
\begin{alignedat}{2}
&(J_{\eta}\ast u_\eta)\cdot \nabla n_\eta \rightharpoonup u_\varepsilon\cdot \nabla n_\varepsilon
\quad &\text{in}\quad &\mathcal{D}'(Q_T),\\
&(J_{\eta}\ast u_\eta)\cdot \nabla c_\eta \rightharpoonup u_\varepsilon\cdot \nabla c_\varepsilon
\quad &\text{in}\quad &\mathcal{D}'(Q_T),\\
&(J_{\eta}\ast u_\eta)\cdot \nabla u_\eta \rightharpoonup u_\varepsilon\cdot \nabla u_\varepsilon
\quad &\text{in}\quad &\mathcal{D}'(Q_T).
\end{alignedat}
\end{equation*}
Similarly, from \eqref{weak:app}, \eqref{strongap} and the upper bound of $c_\eta$, let $\eta\to0$, the weak-strong convergence yields
\begin{equation*}
\begin{aligned}
\chi(c_\eta)n_\eta \nabla (J_\varepsilon\ast c_\eta)
&\rightharpoonup \chi(c_\varepsilon)n_\varepsilon  \nabla (J_\varepsilon\ast c_\varepsilon)
&\quad \text{in}\quad \mathcal{D}'(Q_T),\\
n_\eta f(c_\eta)
&\rightharpoonup n_\varepsilon  f(c_\varepsilon)
&\quad \text{in}\quad \mathcal{D}'(Q_T).
\end{aligned}
\end{equation*}
After extracting a subsequence, since $n_\eta$ converges to $n_\varepsilon$ a.e. in $(0,T)\times K$ for any compact subset $K$ of $\mathbb{R}^d$, Egorov's theorem indicates that for any given $\delta>0$, there exists a subset $Q_\delta\subset (0,T)\times K$ such that it satisfies $|((0,T)\times K)\setminus Q_\delta|<\delta$ and $n_\eta$ converges to $n_\varepsilon$ uniformly on $Q_\delta$ as $\eta\rightarrow0$. Recalling that $n_\eta$ is uniformly bounded in $L^{\infty}(0,T;L^q(\mathbb{R}^d))$ for any $1\leq q<\infty$, as $\eta\rightarrow0$, we have
\begin{align*}
&\|n_\eta-n_\varepsilon\|_{L^{2(m-1)}((0,T)\times K)}\\
&\quad\leq \|n_\eta-n_\varepsilon\|_{L^{2(m-1)}(Q_\delta)}+\|(n_\eta,n_\varepsilon)\|_{L^{\infty}(0,T;L^{2m})} T^{\frac{1}{2m}} |((0,T)\times K)\setminus Q_\delta|^{\frac{1}{2m(m-1)}}\\
&\quad\leq \|n_\eta-n_\varepsilon\|_{L^{2(m-1)}(Q_\delta)}+C_{T,\varepsilon} \delta^{\frac{1}{2m(m-1)}}\\
&\quad\rightarrow C_{T,\varepsilon}\delta^{\frac{1}{2m(m-1)}}.
\end{align*}
As $\delta$ can be arbitrary, this in particular implies that as $\eta\to0$, $n_\eta^{m-1}$ converges to $n_\varepsilon^{m-1}$ strongly in $L^2(0,T;L^2_{\rm{loc}}(\mathbb{R}^d))$. At the moment, $J_{\eta}\ast n_\eta^{m-1}$ converges to $n_\varepsilon^{m-1}$ strongly in $L^2(0,T;L^2_{\rm{loc}}(\mathbb{R}^d))$ as $\eta\to0$. Consequently, using the weak-strong convergence, one has
\begin{align*}
\big( n_\eta^{m-1}\ast J_{\eta}\big)\nabla n_\eta\rightharpoonup n_\varepsilon^{m-1} \nabla n_\varepsilon\quad\text{in}\quad \mathcal{D}'(Q_T)\text{ as }\eta\to0.
 \end{align*}
The above convergence properties imply that $(n_\varepsilon,c_\varepsilon,u_\varepsilon)$ is indeed a global weak solution to the problem \eqref{app}-\eqref{appd} which obeys \eqref{r100}. The proof of Proposition \ref{propapp} is finished.

\section{Preliminary lemmas}\label{appendix2}
\begin{lemma}[cf.~\cite{gn1983}]\label{ei}Let $\Omega$ be a bounded domain in $\mathbb{R}^d$ with Lipschitz boundary and $p^*:=\frac{dp}{d-p}$ with $1\leq p<d$. Then there is a positive constant $c$, depending on $d$, such that 
\[\|u-u_S\|_{L^{p^*}(\Omega)}\leq \frac{c D^{d+1-\frac{d}{p}}}{|S|^{\frac{1}{p}}}\|\nabla u\|_{L^p(\Omega)},\quad \forall\ u\in W^{1,p}(\Omega),\]
where $S$ is any measurable subset of $\Omega$ with $|S|>0$, $u_S=\frac{1}{|S|}\int_{S}udx$, and $D$ is the diameter of $\Omega$.
\end{lemma}

\begin{lemma}[Gagliardo-Nirenberg-Sobolev inequality, cf. \cite{N}]\label{t9}
	Let $d\geq1$, $q,r$ satisfy $1\leq q,r \leq \infty$ and $j,m\in \mathbb{Z}^+$
satisfy $0\leq j<m$. For any $f\in \mathcal{C}_0^{\infty}(\mathbb{R}^d)$, we then have
	\begin{equation*}
	\|D^jf\|_{L^p(\mathbb{R}^d)}\leq
C\|D^mf\|_{L^r(\mathbb{R}^d)}^{\alpha}\|f\|_{L^q(\mathbb{R}^d)}^{1-\alpha},
	\end{equation*}
	where
$\frac{1}{p}-\frac{j}{d}=\alpha(\frac{1}{r}-\frac{m}{d})+(1-\alpha)\frac{1}{q}$, $\frac{j}{m}\leq\alpha\leq1$
and $C>0$ depends on $m,d,j,q,r,\alpha$. 

\end{lemma}

We recall the classical Aubin-Lions-Simon compactness lemma.

\begin{lemma}[Aubin-Lions-Simon compactness lemma, cf.~\cite{simon1986}]\label{lemmaAubin}
Let $\Omega\subset \mathbb{R}^d$ ($d\geq1$) be a bounded domain with $\partial\Omega \in \mathcal{C}^{0,1}$. Let the spaces $X, Y$ and the Banach space $B$ be defined on $\Omega$ and satisfy that $X$ embeds compactly in $B$, which in turn embeds continuously in $Y$.  For some $1\leq p,r\leq \infty$ such that $p<\infty$ or $r>1$, assume that the sequence $\{f^\varepsilon\}_{0<\varepsilon<1}$ is uniformly bounded in $L^p(0,T;X)$ and $\{\partial_t f^\varepsilon\}_{0<\varepsilon<1}$ is uniformly bounded in $L^r(0,T;Y)$. Then 
\begin{itemize}
    \item $\{f^\varepsilon\}_{0<\varepsilon<1}$ is relatively compact in $L^{p}(0,T;B)$.
    \item If $p=\infty$ and $r>1$, then $\{f^\varepsilon\}_{0<\varepsilon<1}$ is relatively compact in $\mathcal{C}([0,T];B)$.
\end{itemize}
\end{lemma}

The Dubinski\"i compactness lemma is useful for proving the compactness of $n$ with nonlinear diffusion.

\begin{lemma}[Dubinski\"i compactness lemma,  cf.~\cite{Dunbinski, carrillo2012cross}]\label{lemmaDub} Let $\Omega\subset \mathbb{R}^d$ {\rm(}$d\geq1${\rm)} be a bounded domain with $\partial\Omega \in C^{0,1}$. Assume that the sequence $\{f^\varepsilon\}_{0<\varepsilon<1}$ satisfies
$$
\|\partial_t f_\varepsilon\|_{L^1(0,T; H^{-s}(\Omega))}+\|f_\var^p\|_{L^  q(0,T;H^1(\Omega))}\leq C,
$$
for some $p, q\geq 1$, $s>0$ and constant $C>0$ independent of $\var$. Then $\{f^\varepsilon\}_{0<\varepsilon<1}$ is relatively compact in $L^{pl}(0,T;L^r(\Omega))$ for any $r<\infty$ and $l<q$.
 \end{lemma}

\end{appendices}

\noindent \textbf{Acknowledgments} 
We are very thankful to the referees for carefully reading
our paper and giving us insightful suggestions/comments,
which greatly improve the  presentation of our manuscript. L.-Y. Shou is supported by National Natural Science Foundation of China (Grant No. 12301275). L. Wu is partially supported by National Natural Science Foundation of China (Grant No. 12401133) and the Guangdong Basic and Applied Basic Research Foundation (2025B151502069).

\vspace{2mm}

\noindent \textbf{Conflict of interest.} The authors do not have any possible conflicts of interest.

\vspace{2mm}

\noindent \textbf{Data availability statement.}
 Data sharing is not applicable to this article, as no datasets were generated or analyzed during the current study.

\bibliographystyle{abbrv} 
\parskip=0pt
\small
\bibliography{MQBH_bib}

\vspace{4mm}
\newpage

\noindent(Q. He)\par\nopagebreak
\noindent\textsc{Sorbonne Universit{\'e}, CNRS, Laboratoire Jacques-Louis Lions, F-75005 Paris, France.}

\noindent Email address: {\tt qyhe.cnu.math@qq.com}

\vspace{3ex}

\noindent (L.-Y. Shou)\par\nopagebreak
\noindent\textsc{School of Mathematical Sciences, Ministry of Education Key Laboratory of NSLSCS, and Key Laboratory of Jiangsu Provincial Universities of FDMTA, Nanjing Normal University, Nanjing, 210023, P. R. China}

\noindent $^{*}$Corresponding author. Email address: {\tt shoulingyun11@gmail.com}

\vspace{3ex}

\noindent(L. Wu)\par\nopagebreak
\noindent\textsc{School of Mathematics, South China University of Technology,
Guangzhou, 510640, P. R. China}

\noindent Email address: {\tt leyunwu@scut.edu.cn}

\end{document}